\documentclass{daj}

\usepackage{ amssymb, amsmath, enumerate, amsfonts, amsthm, mathrsfs, url, bm, mathtools}

 \newcommand{\set}[1]{\left\{#1\right\}}
\newcommand{\bigset}[1]{\bigl\{ #1 \bigr\}}

\newcommand{\abs}[1]{\left| #1\right|}
\newcommand{\bigabs}[1]{\bigl| #1 \bigr|}

\newcommand{\biggabs}[1]{\biggl| #1 \biggr|}
\newcommand{\Biggabs}[1]{\Biggl| #1 \Biggr|}
\newcommand{\sqbrac}[1]{\left[ #1 \right]}

\newcommand{\ceil}[1]{\left\lceil #1 \right\rceil}

\newcommand{\floor}[1]{\left\lfloor #1 \right\rfloor}
\newcommand{\brac}[1]{\left( #1 \right)}
\newcommand{\bigbrac}[1]{\bigl( #1 \bigr)}

\newcommand{\biggbrac}[1]{\biggl( #1 \biggr)}
\newcommand{\Biggbrac}[1]{\Biggl( #1 \Biggr)}
\newcommand{\norm}[1]{\left\| #1\right\|}
\newcommand{\bignorm}[1]{\big\| #1 \big\|}

\newcommand{\ang}[1]{\left\langle#1\right\rangle}

\newcommand{\recip}[1]{\frac{1}{#1}}
\newcommand{\trecip}[1]{\tfrac{1}{#1}}

\newcommand{\N}{\mathbb{N}}
\newcommand{\Z}{\mathbb{Z}}

\newcommand{\R}{\mathbb{R}}
\newcommand{\C}{\mathbb{C}}
\newcommand{\T}{\mathbb{T}}

\newcommand{\E}{\mathbb{E}}

\newcommand{\hcf}{\mathrm{hcf}}
\newcommand{\meas}{\mathrm{meas}}

\newcommand{\intd}{\mathrm{d}}
\newcommand{\supp}{\mathrm{supp}}

\newcommand{\eps}{\varepsilon}
\newcommand{\hash}{\#}

\makeatletter
\let\@@pmod\pmod
\DeclareRobustCommand{\pmod}{\@ifstar\@pmods\@@pmod}
\def\@pmods#1{\mkern4mu({\operator@font mod}\mkern 6mu#1)}
\makeatother

\renewcommand{\leq}{\leqslant}
\renewcommand{\geq}{\geqslant}

\newtheorem{theorem}{Theorem}[section]
\newtheorem{question}[theorem]{Question}
\newtheorem{conjecture}[theorem]{Conjecture}
\newtheorem{corollary}[theorem]{Corollary}

\newtheorem{proposition}[theorem]{Proposition}
\newtheorem{lemma}[theorem]{Lemma}

\theoremstyle{definition}
\newtheorem{definition}[theorem]{Definition}
\newtheorem*{remark}{Remark}

\dajAUTHORdetails{%
  title = {Counting Monochromatic Solutions to Diagonal Diophantine Equations}, %% please capitalize all significant words
  author = {Sean Prendiville},
  plaintextauthor = {Sean Prendiville},
  runningtitle = {Counting Monochromatic Solutions}, 
}   %%% END \dajAUTHORdetails

\dajEDITORdetails{%
   year={2021},
   number={14},
   received={9 October 2020},   % received date: example: 7 January 2017
   published={17 September 2021},  % published date
   doi={10.19086/da.28173},       % XXX = number of paper, e.g. da006 for paper#6
}   %%% END \dajEDITORdetails

\begin{document}

\begin{frontmatter}[classification=text]
%% EDITOR: this will force the keywords to appear right after the Abstract.
%%   If the abstract is too long and would force the keywords off the
%%   front page, please comment out % [classification=text] above
%%   This way the keywords will be floated on the bottom of the first page
%%   even though the Abstract spills over to the next page.

%%% AUTHOR: Title goes here.  This line is optional.  You must use it
%%   if title has footnote attached or requires nontrivial typesetting,
%%   e.g., inclusion of linebreaks to force nice layout.
%\title{Counting Monochromatic Solutions to Diagonal Diophantine Equations} %% please capitalize all significant words

%%% AUTHOR:
%%% List all authors. If you wish, place grant acknowledgements in \thanks.
%%% In brackets include a short tag for each author.
\author[sean]{Sean Prendiville}%\thanks{Supported by...}}
%\author[johan]{Johan H{\aa}stad\thanks{Supported by...}}
%\author[laci]{L\'aszl\'o Lov\'asz\thanks{Supported by...}}
%\author[andy]{Andrew Chi-Chih Yao\thanks{Supported by...}}

%%% AUTHOR: Abstract goes here
\begin{abstract}
We show how to adapt the Hardy--Littlewood circle method to count monochromatic solutions to diagonal Diophantine equations. This delivers a lower bound which is optimal up to absolute constants. The  method is illustrated on equations obtained by setting a diagonal quadratic form equal to a linear form. As a consequence, we determine an algebraic criterion for when such equations are partition regular. Our methods involve discrete harmonic analysis and require a number of `mixed' restriction estimates, which may be of independent interest.
\end{abstract}
\end{frontmatter}

\setcounter{tocdepth}{1}
\tableofcontents

%%% AUTHOR: body of paper starts here
\section{Introduction}\label{introduction}

The Hardy--Littlewood circle method has been tremendously successful in delivering asymptotic estimates for the number of solutions to a given diagonal Diophantine equation \cite{DavenportAnalytic, 
VaughanHardy, VaughanWooleyWaring, WooleyTranslation}, even when variables are constrained to arithmetically structured sets such as primes  \cite{HuaAdditive}. It is less apparent how to count solutions to equations within unstructured sets of integers - those for which we have only combinatorial information. In this article we illustrate an adaptation of the circle method designed to count solutions to an equation with variables  constrained to a cell of a partition. Prototypical of our results is the following counting version of a theorem of Bergelson \cite{BergelsonErgodicUpdate}.

%A substantial portion of Ramsey theory concerns properties which persist under finite partitions, such as the property of solving a pre-determined Diophantine equation. 

% For instance, it is a longstanding problem of Erd\H{o}s and Graham \cite{GrahamSome, GrahamOld} to determine whether the Pythagorean equation is partition regular.
%
%Perhaps the first truly non-linear result is due to Bergelson \cite{BergelsonErgodicUpdate}, asserting partition regularity of the equation $x-y=z^2$.  

%
%We prove a counting version of Bergelson's theorem, which is prototypical of the results of this paper.
\begin{theorem}\label{bergelson lower bound}
For any $r$-colouring $C_1\cup \dots \cup C_r = \{1,2,\dots, N\}$ there exists a colour class $C_j$ such that\footnote{For our conventions regarding asymptotic notation, see \S\ref{notation}.}
\begin{equation}\label{bergelson bound eqn}
\sum_{x-y = z^2} 1_{C_j}(x)1_{C_j}(y)1_{C_j}(z)\gg_r N^{3/2^r}(1-o_r(1)).
\end{equation}
\end{theorem}
The lower bound in \eqref{bergelson bound eqn} is far from the total number of solutions to the equation $x-y=z^2$ in the interval $[N] = \{1,2,\dots, N\}$, which is of order $N^{3/2}$.  However, the order of magnitude in \eqref{bergelson bound eqn} is optimal, as can be seen from the colouring
\begin{equation}\label{bad colouring}
C_1 := (N^{1/2}, N],\quad %C_2 := (N^{1/4}, N^{1/2}],\ 
\dots\quad ,\quad C_{r-1} := (N^{1/2^{r-1}}, N^{1/2^{r-2}}], \quad C_r := [N^{1/2^{r-1}}].
\end{equation}
\begin{proposition}\label{moreira upper bound}
There exists an $r$-colouring of $[N]$ with at most $O(N^{3/2^r})$ monochromatic solutions to the equation $x-y = z^2$.
\end{proposition}
The argument underlying Theorem \ref{bergelson lower bound} utilises the Fourier-analytic regularity lemma of Green \cite{GreenSzemeredi}, which has a convenient formulation due to Green and Tao \cite{GreenTaoArithmetic}. The robustness of the regularity lemma allows us to prove the following generalisation of Theorem \ref{bergelson lower bound}.
\begin{theorem}[Linear counting theorem]\label{linear form satisfies rado}
Let $a_1,\dots, a_s, b_1,\dots, b_t \in \Z\setminus\set{0}$  and suppose that there exists $I \neq \emptyset$ such that $\sum_{i \in I}a_i = 0$. For any $r$-colouring $C_1 \cup \dots \cup C_r = [N]$ there exists a colour class $C_j$ such that
\begin{equation}
\sum_{a_1x_1 + \dots + a_s x_s = b_1y_1^2 + \dots + b_ty_t^2} 1_{C_j}(x_1) \dotsm 1_{C_j}(x_s)1_{C_j}(y_1) \dotsm 1_{C_j}(y_t)\\ \gg_{ r} N^{(|I|+s+t-2)/2^r}(1-o_r(1)).
\end{equation}
(Here we have suppressed the dependence of implicit constants on $a_i$ and $b_j$.)
\end{theorem}
Turning to equations without linear terms, Chow, Lindqvist and the author \cite{CLPRado} have classified when diagonal homogeneous equations
\begin{equation}\label{diagonal quadric}
a_1x_1^k + \dots + a_sx_s^k = 0
\end{equation}
have solutions in some cell of a finite partition.
\begin{definition}[Partition regular]
We say that the equation $P(x_1, \dots, x_s) = 0$ is \emph{partition regular} if for any finite partition of the positive integers $\N = C_1 \cup \dots \cup C_r$ there exists $C_j$ and infinitely many $(x_1, \dots, x_s) \in C_j^s$ such that $P(x_1, \dots, x_s) = 0$. One may think of a partition into $r$ parts as a colouring with $r$ colours, in which case we call $(x_1, \dots, x_s)$ a \emph{monochromatic solution}.
\end{definition}
Rado \cite{RadoStudien} completely characterised which linear forms $P = a_1x_1+ \dots + a_sx_s$ are partition regular: it is both necessary and sufficient that there exists $I \neq \emptyset$ such that $\sum_{i \in I} a_i = 0$. The criterion for \eqref{diagonal quadric} is identical, subject to the caveat\footnote{The Fermat cubic illustrates that some such caveat is necessary.} that the number of variables $s$ is sufficiently large in terms of the degree $k$. For squares ($k=2$) we require $s \geq 5$ at present.  The methods of \cite{CLPRado} do not yield a lower bound on the number of monochromatic solutions to \eqref{diagonal quadric}, though it was conjectured \cite[\S3.1]{CLPRado} that such a result should be true.  The original motivation for the present paper is to settle this conjecture affirmatively.
\begin{theorem}\label{quadric counting}
Let $a_1, \dots, a_s \in \Z\setminus\set{0}$ with $s \geq  5$.  Suppose that there exists $I \neq \emptyset$ such that $\sum_{i \in I} a_i = 0$.  Then for any for any $r$-colouring $[N] = C_1 \cup \dots \cup C_r$ there exists $C \in \set{C_1, \dots, C_r}$ such that
$$
\sum_{a_1x_1^2 + \dots + a_sx_s^2 = 0}  1_C(x_1)\dotsm1_C(x_s)  \gg_{r} N^{s - 2}(1-o_{r}(1)).
$$
(Here we have suppressed the dependence of implicit constants on the coefficients $a_i$.)
\end{theorem}
A standard application of the circle method (see \cite{VaughanHardy}) shows that the bound in Theorem \ref{quadric counting} is optimal.
\begin{proposition}
For any $a_1, \dots, a_s \in \Z\setminus\set{0}$ with $s \geq  5$ we have the upper bound
$$
\sum_{a_1x_1^2 + \dots + a_sx_s^2 = 0} 1_{[N]}(x_1)\dotsm1_{[N]}(x_s) \ll N^{s - 2}.
$$
\end{proposition}
%
%\begin{remark}[Higher degree diagonal equations]
%The methods of this paper also yield a counting result for higher degree diagonal equations of the form \eqref{diagonal quadric}. We  restrict our attention to squares for a simpler exposition.
%\end{remark}
 One consequence of celebrated work of Moreira \cite{MoreiraMonochromatic} is partition regularity of the equation
\begin{equation}\label{moreira eqn}
a_1x_1^2 + \dots + a_s x_s^2 = x_0,
\end{equation}
under the assumption that 
\begin{equation}\label{translation assumption}
a_1 + \dots + a_s = 0.
\end{equation}  
Moreira's methods are inductive, and locate a monochromatic solution arising from a special two-parameter subvariety. To obtain a counting result for \eqref{moreira eqn} by modifying these methods  seems unlikely to be possible. Using an alternative approach, we obtain a counting result for \eqref{moreira eqn} and in addition are able to substantially relax the assumption \eqref{translation assumption} on the coefficients. The price we pay for this strengthening is that we must assume the quadratic form has sufficiently many variables. 
\begin{theorem}[Quadratic counting theorem]\label{quadratic form satisfies rado}
Let $a_1, \dots, a_s \in \Z\setminus\set{0}$ and $b_1, \dots, b_t \in \Z\setminus\set{0}$ with $s \geq 3$ and $s+t \geq 5$.  Suppose that there exists $I \neq \emptyset$ such that $\sum_{i \in I} a_i = 0$.  Then for any  $r$-colouring $[N] = C_1 \cup \dots \cup C_r$ there exists $C \in \set{C_1, \dots, C_r}$ such that
$$
\sum_{a_1x_1^2 + \dots + a_sx_s^2 = b_1y_1 + \dots + b_t y_t} \prod_i 1_C(x_i) \prod_j 1_C(y_j) \gg_{r} N^{s+t - 2}(1-o_{r}(1)).
$$
(Here we have suppressed the dependence of implicit constants on $a_i$ and $b_j$.)
\end{theorem}
All equations so far considered   have the form 
\begin{equation}\label{linear equals quadratic}
a_1x_1^2 + \dots + a_sx_s^2 = b_1y_1 + \dots + b_t y_t,
\end{equation}
where the $a_i$ and $b_j$ are non-zero integers. Another equation of this type, $x+y = z^2$, has received attention from Green--Lindqvist \cite{GreenLindqvistMonochromatic} and Pach \cite{PachMonochromatic}.  They demonstrate that $x+y = z^2$ has infinitely many monochromatic solutions in any 2-colouring, but that there is a 3-colouring with no monochromatic solutions beyond $(x,y,z) = (2,2,2)$.   With this in mind, it is natural to ask the following.
\begin{question}
Let $a_1,\dots, a_s, b_1,\dots, b_t \in \Z\setminus\set{0}$ with $s,t\geq 1$. When is the equation \eqref{linear equals quadratic} partition regular?
\end{question}
Ideally we would like an algebraic characterisation  comparable  to that of  \cite{RadoStudien} and \cite{CLPRado}, a criterion which can be easily checked by a computer. A necessary condition is provided in work of Di Nasso and Luperi Baglini \cite[Theorem 3.10]{DiNassoLuperiBagliniRamsey}.
\begin{proposition}[Di Nasso and Luperi Baglini]\label{necessary}
If the equation \eqref{linear equals quadratic} is partition regular, then there exists $I \neq \emptyset$ such that either $\sum_{i \in I} a_i = 0$ or $\sum_{i \in I } b_i = 0$.
\end{proposition}

We are able to show that this condition is sufficient in all but one case.
\begin{theorem}[Linear--quadratic partition regularity]\label{criteria} 
Let $a_1,\dots, a_s, b_1,\dots, b_t \in \Z\setminus\set{0}$ with $s,t\geq 1$.
Suppose that  \eqref{linear equals quadratic} does not take the form
\begin{equation}\label{bad case}
a(x_1^2 - x_2^2) = by^2 + cz
\end{equation}
for some non-zero integers $a,b,c$. 
Then \eqref{linear equals quadratic} is partition regular if and only if there exists $I \neq \emptyset$ such that $\sum_{i \in I} a_i = 0$ or $\sum_{i \in I} b_i = 0$.
\end{theorem}
This almost resolves \cite[Open Problem 1] {DiNassoLuperiBagliniRamsey} when the question is restricted to the family of Diophantine equations given by \eqref{linear equals quadratic}. Our lack of knowledge regarding \eqref{bad case} is an artefact of our methods. We believe that Di Nasso and Luperi Baglini's criterion is the correct characterisation.
\begin{conjecture}
For any non-zero integers $a,b,c$, the equation \eqref{bad case} is partition regular.
\end{conjecture}
As evidence towards this conjecture, we prove that a special case of \eqref{bad case} is partition regular conditional on the following notorious problem of Hindman.
\begin{conjecture}[Hindman]\label{hindman}
In any finite colouring of $\N$ there is a monochromatic configuration of the form $\set{x, y, x+y, xy}$.
\end{conjecture}

\begin{theorem}\label{bad under hindman}
If Hindman's conjecture is true, then the equation 
\begin{equation}\label{bad special case}
x_1^2 - x_2^2 = y^2 + z
\end{equation}
is partition regular.
\end{theorem}

There are too few variables for our methods to handle \eqref{bad case}. Although the number of variables is more than sufficient for the circle method to count solutions to \eqref{bad case} with $x_1, x_2, y \in [N]$ and $z \in [N^2]$, the possibility of colourings such as \eqref{bad colouring} means that we are driven to counting solutions with all variables (including $z$) constrained to the interval $[N]$. Standard circle method heuristics indicate that such a count is at the threshold of `square-root cancellation', and necessitates a more refined analysis, such as that given by Kloosterman \cite{KloostermanRepresentation} or Heath--Brown \cite{HeathBrownNew}. Incorporating combinatorial methods within such an analysis remains a challenging problem.

\subsection{Higher degree equations}

The methods of this paper also yield counting results for higher degree diagonal equations.  For instance, one can adapt our argument to show that in any $r$-colouring of $[N]$ the number of monochromatic solutions to  \eqref{diagonal quadric} is $\gg_{r,k, a_i} N^{s-k}$, provided that $s$ is sufficiently large in terms of $k$. Hence, just as with diagonal quadrics, the number of monochromatic solutions  is a positive proportion of the total number of solutions in an interval. 

Similarly, one can use our methods to classify when inhomogeneous equations involving diagonal forms of other degrees are partition regular, provided there are sufficiently many variables. We have focused on diagonal linear/quadratic  forms for  simplicity.

\subsection{Mixed restriction estimates}

The main tools used in proving our results are the Hardy--Littlewood circle method, the abelian arithmetic regularity lemma and the Fourier analytic transference principle. All three of these tools are part of discrete harmonic analysis, and key to their success are so-called \emph{discrete restriction estimates}\footnote{See the introduction to \cite{HenriotHughesRestriction} for motivation and history.}. 

Colourings such as \eqref{bad colouring}, when combined with the inhomogeneity of the equation \eqref{linear equals quadratic}, force us to count solutions to equations in certain `skewed' regions, where some variables are constrained to much smaller intervals than is typical in the circle method. This necessitates the development of some novel `mixed' restriction estimates (see Lemma \ref{mixed restriction}), such as the following.
\begin{theorem}[Mixed restriction]
Let $W$ be a positive integer and $p > 2$. Then either $N \ll_p W^{O_p(1)}$ or, for any $f, g : \Z \to \C$ with $|f|, |g| \leq 1_{[N]}$ we have 
\begin{equation}\label{baby mixed restriction}
  \int_{\T} \abs{\sum_{N/2 < x\leq N} f(x) e(W\alpha x^2)\sum_{N/2 < y\leq N} g(y) e(\alpha y)}^p \intd\alpha \\
  \ll_{p} N^{2p - 2} W^{-1}.
  \end{equation}
\end{theorem}
We note that 
\begin{equation*}
  \int_{\T} \abs{\sum_{N/2 < x\leq N} f(x) e(W\alpha x^2)}^{2p} \intd\alpha \\
  \ll_{p} N^{2p - 2} 
  \end{equation*}
  and
\begin{equation*}
  \int_{\T} \abs{\sum_{N/2 < y\leq N} g(y) e(\alpha y)}^{2p} \intd\alpha \\
  \ll_{p} N^{2p - 1} .
  \end{equation*}
Hence the obvious application of the Cauchy--Schwarz inequality does not deliver a bound as strong as \eqref{baby mixed restriction}.

In addition to \eqref{baby mixed restriction}, we require three further mixed restriction estimates, and to prove all four simultaneously we abstract an approach of Bourgain \cite{BourgainLambda}. Hence, in \S\ref{abstract restriction section} we prove a general restriction estimate for exponential sums obeying certain hypotheses and in \S\ref{Fourier Pseudorandomness} we verify that each of our four mixed exponential sums satisfy these hypotheses.

\subsection{The utility of counting results}
 In the study of partition regularity it is often desirable to delineate between `trivial' and `non-trivial' solutions to an equation, as some equations possess monochromatic solutions for uninteresting reasons. For instance $x+y = z^2$  has the solution $(2,2,2)$, whilst $x + y = 2z$ is always solved by the diagonal $(x,x,x)$.  One commonly encountered choice of non-triviality is a solution in which all variables are distinct, but the precise notion may depend on the application. A counting result allows one to ensure the existence of monochromatic solutions avoiding \emph{any} sparse subset of solutions. This implies that there are monochromatic solutions of `generic type', i.e.\ not lying on a proper Zariski closed subset. For if all monochromatic solutions took this form then counting arguments would likely give a power saving in the number of monochromatic solutions when compared with the total number of solutions.

Frankl, Graham and R\"odl \cite{FGRQuantitative} pioneered the counting of monochromatic solutions to systems of \emph{linear} equations, obtaining lower bounds of the correct order of magnitude for all such partition regular systems. %They showed that in any finite colouring of the interval $[N]$ the number of monochromatic solutions is a positive proportion of the total number of solutions in the interval\footnote{The proportion depends on the number of colours, but is otherwise uniform over all $r$-colourings.}. 
The non-linear theory is much less developed, mainly due to our lack of knowledge regarding when such equations are partition regular. The author hopes this paper encourages the development of further non-linear counting results.

\subsection{Organisation of this paper}

We sketch some of the ideas behind our methods in \S\ref{sketch section}. In \S\ref{linear density section} we use the arithmetic regularity lemma to prove that dense sets of integers contain certain polynomial configurations, from which all of our counting results are ultimately derived. We derive Theorem \ref{linear form satisfies rado} from the results of \S\ref{linear density section} in \S\ref{linear colouring subsection}.  

We devote \S\ref{abstract restriction section}--\ref{quadratic density sec} to modifying the results of \S\ref{linear density section} to apply to dense sets of \emph{squares}, instead of just dense sets of integers. In \S\ref{abstract restriction section} we generalise an approach of Bourgain \cite{BourgainLambda} to prove a general restriction estimate for exponential sums obeying certain hypotheses and in \S\ref{Fourier Pseudorandomness} we verify these hypotheses for the exponential sums of relevance. In \S\ref{control section} we use these restriction estimates to show how the Fourier transform of a set completely determines the number of solutions it contains to the equations we are interested in.

All of our counting results are derived from density results in \S\ref{colouring section}. Finally in \S\ref{moreira lindqvist} we adapt an argument of Moreira to establish partition regularity of equations of the form \eqref{linear equals quadratic} which are not covered by our counting theorems. This allows us to combine all previous results to deduce our partition regularity criteria (Theorem \ref{criteria}).

\subsection{Notation}\label{notation}

\subsubsection*{Standard conventions}
We use $\N$ to denote the positive integers.  For a real number $X \geq 1$, write $[X] = \{ 1,2, \ldots, \floor{X}\}$.  A complex-valued function is said to be \emph{1-bounded} if the modulus of the function does not exceed 1. 

We use counting measure on $\Z$, so that for $f,g :\Z \to \C$, we have
$$
\norm{f}_{L^p} := \biggbrac{\sum_x |f(x)|^p}^{\recip{p}}, \ \ang{f,g} := \sum_x f(x)\overline{g(x)},\ \text{and}\ (f*g)(x) := \sum_y f(y)g(x-y).
$$ 
Any sum of the form $\sum_x$ is to be interpreted as a sum over $\Z$. The \emph{support} of $f$ is the set $\supp(f) := \set{x \in \Z : f(x) \neq 0}$. We write $\norm{f}_\infty$ for $\sup_{x} |f(x)|$.

For a finite set $S$ and function $f:S\to\C$, denote the average of $f$ over $S$ by
\[
\E_{s\in S}f(s):=\frac{1}{|S|}\sum_{s\in S}f(s).
\]
 
We use Haar probability measure on $\T := \R/\Z$, so that for integrable $F, G : \T \to \C$, we have
\begin{multline*}
\norm{F}_{L^p} := \biggbrac{\int_\T |F(\alpha)|^pd\alpha}^{\recip{p}} = \biggbrac{\int_0^1 |F(\alpha)|^pd\alpha}^{\recip{p}}, \\
 \ang{F,G} := \int_\T F(\alpha)\overline{G(\alpha)}\intd\alpha,\quad \text{and} \quad (F*G)(\alpha) := \int_\T F(\alpha - \beta)G(\beta) \intd\beta.
\end{multline*}
We write $\norm{\alpha}_\T$ for the distance from $\alpha \in \R$ to the nearest integer
$
\min_{n \in \Z} |\alpha - n|.
$
This remains well-defined on $\T$.

\begin{definition}[Fourier transform]\label{fourier transform}
For $f : \Z^d \to \C$ with finite support define $\hat{f} : \T^d \to \C$ by
$$
\hat{f}(\alpha) := \sum_{n\in \Z^d} f(n) e(\alpha \cdot n).
$$
Here $e(\beta)$ stands for $e^{2\pi i \beta}$. We sometimes write $e_q(a)$ for $e(a/q)$.

Given integrable $F : \T^d \to \C$ write
\[
\hat{F}(n) := \int_{\T^d} F(\alpha) e(-\alpha \cdot n) \intd\alpha.
\]
\end{definition}

\begin{definition}[Smooth/rough numbers]
We say that an integer $n$ is $w$-smooth if all of its prime divisors are at most $w$.  We say $n$ is $w$-rough if all of its prime divisors are at least $w$.
\end{definition}

\subsubsection*{Asymptotic notation}
For a complex-valued function $f$ and positive-valued function $g$, write $f \ll g$ or $f = O(g)$ if there exists a constant $C$ such that $|f(x)| \le C g(x)$ for all $x$. We write $f = \Omega(g)$ if $f \gg g$.  The notation $f\asymp g$ means that $f\ll g$ and $f\gg g$.  We write $f = o(g)$ if for any $\eps > 0$ there exists $X\in \R$ such that for all $x \geq X$ we have $|f(x)| \leq \eps g(x)$. 
 
Subscripts are added to this notation to indicate when the implicit constant/function depends on other parameters. For instance, the statement of Theorem \ref{bergelson lower bound} can be phrased as: For any positive integer $r$ there exists a constant $c_r >0$ and a function $\tau_r : \Z_{>0} \to [0,\infty)$ with $\tau_r(N) \to 0$ as $N\to \infty$ such that the following holds. For any $r$-colouring $C_1\cup \dots \cup C_r = \set{1,\dots, N}$ there exists $C_j$ such that
$$
\sum_{x-y = z^2} 1_{C_j}(x)1_{C_j}(y)1_{C_j}(z)\geq c_r N^{3/2^r}(1-\tau_r(N)).
$$
\subsubsection*{Local conventions}
The following are idiosyncratic to this paper, and may not be adhered to elsewhere in the literature. 
\begin{definition}[Quadratic Fourier transform]\label{quadratic fourier transform}  Given $f : \Z \to \C$ with finite support, define the \emph{quadratic Fourier transform} by
$$
\tilde{f}(\alpha) := \sum_{x } f(x) e(\alpha x^2).
$$ 
\end{definition}

\begin{definition}[Non-singular linear form]\label{linear form}
Let $c_1, \dots, c_s \in \Z$.  We call a polynomial of the form
$$
L(x_1, \dots, x_s) = c_1x_1 + \dots + c_s x_s
$$ 
a \emph{linear form}.  We say the linear form is \emph{non-singular} if $c_i \neq 0$ for all $i$. If $x = (x_1, \dots, x_s) \in \Z^s$, then it will be convenient to use the shorthand
$$
L(x^2) := L(x_1^2, \dots, x_s^2).
$$
\end{definition}

\begin{remark}[Dependence of implicit constants on linear forms]
A number of results in the remainder of the paper concern three non-singular linear forms $L_1, L_2, L_3$. Throughout we suppress dependence of implicit constants on the number of variables and the coefficients of the $L_i$.  One may think of all data associated to the $L_i$ as being $O(1)$.
\end{remark}

\section{A sketch of our methods}\label{sketch section}

As with the author's previous two papers on partition regularity \cite{CLPRado, ChapmanPrendivilleRamsey}, we first exhibit the method underlying our results with a proof of Schur's theorem.
\subsection{The regularity approach to Schur's theorem}
\begin{theorem}[Schur]\label{schur}
For any $r$-colouring $C_1\cup \dots \cup C_r = \{1,2,\dots, N\}$ there exists a colour class $C_j$ and $x,y,z \in C_j$ such that $x+y = z$, provided that $N$ is sufficiently large in terms of $r$.
%such that
%\begin{equation}\label{schur bound}
%\sum_{x+y = z} 1_{C_j}(x)1_{C_j}(y)1_{C_j}(z)\gg_r N^{2}(1-o_r(1)).
%\end{equation}
\end{theorem}
We sketch a proof of this using the Fourier-analytic regularity lemma (Lemma \ref{arithmetic regularity}) originating in \cite{GreenSzemeredi}.
The take-away of the regularity lemma is that we can find a Bohr set
\begin{equation}\label{bohr set}
B := \set{ x\in [N] : \norm{\alpha_i x}_\T \leq \eta \text{ for } i = 1, \dots, d}
\end{equation}
such that each colour class $C_j$ is approximately invariant under shifts by $B$, so that for $y \in B$ we have
\begin{equation}\label{approximate invariance}
1_{C_j}(x+y) \approx 1_{C_j}(x).
\end{equation}
We have been deliberately vague about the nature of the approximation in \eqref{approximate invariance}.  There is an important trade-off to keep in mind: the closer one wishes the approximation \eqref{approximate invariance}, the smaller the resulting Bohr set \eqref{bohr set}.  The nature of the approximation \eqref{approximate invariance} allows us to conclude that for any colour classes $C_i$ and $C_j$ we have
\begin{equation}\label{invariant count}
\sum_{x \in [N]} \sum_{y \in B} 1_{C_i}(x)1_{C_j}(y) 1_{C_i}(x+y) \approx \sum_{x \in [N]}1_{C_i}(x)^2 \sum_{y \in B} 1_{C_j}(y).
\end{equation}
Using Cauchy--Schwarz the right-hand side of \eqref{invariant count} is at least
\[
N^{-1}|C_i|^2|C_j \cap B|.
\]
By the pigeon-hole principle there exists a colour class $C_j$ with $|C_j\cap B| \geq |B|/r$ and hence for all $i$ we have
\begin{equation}\label{naive lower bound}
\sum_{x \in [N]}1_{C_i}(x)^2 \sum_{y \in B} 1_{C_j}(y) \geq  \frac{|C_i|^2 |B|}{rN} .
\end{equation}
The obvious strategy is to now take $i := j$ in \eqref{invariant count} and \eqref{naive lower bound}, to yield
\begin{equation}\label{naive approach}
\sum_{x \in [N]} \sum_{y \in B} 1_{C_j}(x)1_{C_j}(y) 1_{C_j}(x+y) \approx \sum_{x \in [N]}1_{C_j}(x)^2 \sum_{y \in B} 1_{C_j}(y)\geq  \frac{|C_j|^2 |B|}{rN}.
\end{equation}
The drawback with this approach is that the error term in \eqref{invariant count} is of the form $\eps N |B|$. Hence in order to use \eqref{naive approach} to deduce the existence of a monochromatic solution to $x+y = z$, we need the lower bound in \eqref{naive approach} to be of order $N|B|$. This may not happen: imagine the situation in which the colour class $C_j$ is equal to the Bohr set $B$ (for the purposes of this sketch, $|B|$ should be thought of as $o(N)$). The problem we have encountered is that the colour class $C_j$ which is good for the regularity lemma (as it has large intersection with the Bohr set $B$) may not be a dense colour class (which we need for the lower bound in \eqref{naive lower bound} to be useful).

Our solution to this problem is twofold. By adapting the regularity argument outlined above, we first prove an asymmetric  version of Schur's theorem.
\begin{theorem}[Asymmetric Schur]\label{asymmetric schur}
Let $\delta > 0 $ and $A_1, \dots, A_s \subset [N]$ each with $|A_i| \geq \delta N$. Then for any colouring $[N] = C_1 \cup \dots \cup C_r$ there exists a colour class $C_j$ such that for any $A_i$ we have
\[
\sum_{x+y =z} 1_{A_i}(x)1_{C_j}(y)1_{A_i}(z) \gg_{\delta, r, s} N^2(1-o_{\delta, r, s}(1)).
\]
\end{theorem} 
Next, in order to deduce Schur's theorem from this asymmetric version, we `cleave' colour classes into those which are dense and those which are sparse.  Fix a function $\mathcal{F}(M) \to \infty$. A variant of the pigeon-hole principle (see \S\ref{colouring section}) allows us to find a density $1/M$ with $M = O_{r,\mathcal{F}}(1)$ such that for every colour class $C_i$ one of the following holds:
\begin{itemize}
\item either $C_i$ is $1/M$ dense, in that $|C_i|\geq  N/M$;
\item or $C_i$ is $1/\mathcal{F}(M)$ sparse, in that $|C_i| < N/\mathcal{F}(M)$.
\end{itemize}
We have `cleaved', in that we have found a threshold parameter $M$ such that each colour class is either extremely dense in terms of $M$, or extremely sparse in terms of $M$, there are no intermediate colour classes. Notice that if all colour classes have positive upper density, then on taking $M$ sufficiently large there are no sparse colour classes. The cleaving procedure is designed to isolate those colour classes which are sparse at all scales, having  cardinality $o(N)$ in $[N]$ as $N \to \infty$.

Having cleaved, we apply our asymmetric Schur theorem, taking the sets $A_i$ to be those colour classes which are $1/M$ dense.  This yields a colour class $C_j$ such that for any $1/M$ dense colour class $C_i$ we have
\[
\sum_{x+y =z} 1_{C_i}(x)1_{C_j}(y)1_{C_i}(z) \gg_{M, r} N^2.
\]

We would like to take $i = j$ in the above, but we can only do this if $C_j$ is $1/M$ dense. Let us see why this is so. A counting argument shows that
\[
|C_j|N \geq \sum_{x+y =z} 1_{C_i}(x)1_{C_j}(y)1_{C_i}(z).
\]
Hence 
\begin{equation}\label{cj lower bound}
|C_j| \gg_{M,r} N.
\end{equation}
Provided we have chosen our growth function $\mathcal{F}$ so that the implicit constant in \eqref{cj lower bound} is larger than $1/\mathcal{F}(M)$, we deduce that $C_j$ is not $1/\mathcal{F}(M)$ sparse, hence it must be $1/M$ dense, by cleaving.

\subsection{Adapting this to Bergelson's theorem}

Using quadratic Bohr sets in place of Bohr sets, it is relatively simple to adapt the regularity argument underlying Theorem \ref{asymmetric schur} to prove the following.
\begin{theorem}[Asymmetric Bergelson]\label{asymmetric bergelson}
Let $\delta > 0 $ and $A_1, \dots, A_s \subset [N]$ each with $|A_i| \geq \delta N$. Then for any colouring $[N^{1/2}] = C_1 \cup \dots \cup C_r$ there exists a colour class $C_j$ such that for any $A_i$ we have
\begin{equation}\label{asymmetric bergelson bound}
\sum_{x-y =z^2} 1_{A_i}(x)1_{A_i}(y)1_{C_j}(z) \gg_{\delta, r, s} N^{3/2}(1-o_{\delta, r, s}(1)).
\end{equation}
\end{theorem} 
The problem now is how to cleave. Notice that \eqref{asymmetric bergelson bound} counts $z \in C_j\cap[ N^{1/2}]$, and the density/sparsity of $C_j$ on the interval $[N^{1/2}]$ may be independent of the density/sparsity of $C_j$ on $[N]$ (see the colouring \eqref{bad colouring}). To overcome this we find $M = O_{r,\mathcal{F}}(1)$ and scales $X,X_1, \dots, X_r$ with $X_i \geq X^2$ such that $C_i$ is $1/M$ dense on $[X_i]$ if it is $1/\mathcal{F}(M)$ dense on $[X]$. Averaging, there is a translate $a_i + [X^2]$ such that if $C_i$ is $1/\mathcal{F}(M)$ dense on $[X]$ then $C_i$ is $1/M$ dense on $a_i + [X^2]$.  We then take
\[
A_i := \set{x \in [X^2] : a_i + x \in C_i}
\]
in Theorem \ref{asymmetric bergelson}, and apply a similar argument to that given for Schur's theorem.

We note that key to the success of this strategy is the translation invariance of the linear form $x- y$, in that 
\[
(x+a) - (y+a) = z \quad\text{iff} \quad x - y = z.
\]
This is a property enjoyed by any linear form whose coefficients sum to zero. Unfortunately, the same is not true of a quadratic form whose coefficients sum to zero. Overcoming this is the subject of the next subsection.

\subsection{Linearisation via transference}

To prove Theorem \ref{quadratic form satisfies rado}, when the coefficients of the quadratic form satisfy Rado's criterion, we combine our `cleaving' strategy with the following asymmetric density-colouring result.

\begin{theorem}[Quadratic density--colouring result]\label{special quad}
Let $\delta >0$ and let $r$ be a positive integer. For any sets of integers $A_1, \dots, A_s \subset [N]$  each satisfying $|A_i| \geq \delta N$ and for any $r$-colouring $B_1 \cup \dots \cup B_r = [N]$ there exists $B \in \set{B_1, \dots, B_r}$ such that for all $A\in \set{A_1, \dots, A_s}$ we have
\begin{equation*}
\sum_{x_1^2-x_2^2 = y^2 + z_1  + z_2} 1_A(x_1)1_A(x_2)1_{B}(y)1_{B}(z_1)1_{B}(z_2)  \gg_{ \delta, r,s}  N^{3  }(1-o_{\delta, r, s}(1)).
\end{equation*}
\end{theorem}
This is a representative special case of Theorem \ref{quadratic-linear density}, which we have stated for  simplicity.
Using a Fourier analytic transference principle (see \cite{PrendivilleFour}), we deduce Theorem \ref{special quad} from a \emph{linear} density--colouring result, where we have removed the squares from the $x_i$ variables.
\begin{lemma}[Linear density--colouring result]\label{special lin}
Let $\delta >0$ and let $r$ be a positive integer. For any sets of integers $A_1, \dots, A_s \subset [N^2]$  each satisfying $|A_i| \geq \delta N^2$ and for any $r$-colouring $B_1 \cup \dots \cup B_r = [N]$ there exists $B \in \set{B_1, \dots, B_r}$ such that for all $A\in \set{A_1, \dots, A_s}$ we have
\begin{equation*}
\sum_{x_1-x_2 = y^2 + z_1  + z_2} 1_A(x_1)1_A(x_2)1_{B}(y)1_{B}(z_1)1_{B}(z_2)  \gg_{ \delta, r,s}  N^{5  }(1-o_{\delta, r, s}(1)).
\end{equation*}
\end{lemma}
This is superficially similar to the strategy employed in \cite{CLPRado}, but without the presence of the strongly structured `homogeneous sets' (more properly termed multiplicatively syndetic sets, see \cite{ChapmanPartition}). The lack of such structure presents additional obstacles too technical to discuss here. We refer the interested reader to \S\ref{quadratic density sec}.

\section{A linear density result}\label{linear density section}

The aim of this section is to count solutions to equations of the form \eqref{linear equals quadratic} when certain linear variables are constrained to dense sets, and the remaining variables are constrained to a colouring. We eventually use this density result to derive both our linear counting result (Theorem \ref{linear form satisfies rado}) and our quadratic counting result (Theorem \ref{quadratic form satisfies rado}).  Before stating this we remind the reader of our conventions (Definition \ref{linear form}) regarding linear forms.

\begin{theorem}[Linear density result]\label{linear density}
Let $L_1, L_2, L_3$ denote non-singular linear forms, each in $s_i$ variables with $s_1\geq 2$ and $s_1 + s_2 \geq 3$ (we allow for $s_2 = 0$ or $s_3=0$). Suppose that $L_1(1, \dots, 1) = 0$.
For any $\delta >0$ and positive integer $r$, there exists $\eta \gg_{ r,\delta} 1$ such that for any positive integers $W$ and $N$, either $N\ll_{\delta, r, W}1$ or the following holds. Suppose that $W=1$ or $s_3 > 0$.  Then for any sets $A_1, \dots, A_r \subset [N]$  with $|A_i| \geq \delta N$ for all $i$, and any $r$-colouring $C_1 \cup \dots \cup C_r = [\eta (N/W)^{1/2}, (N/W)^{1/2}]$, there exists $C_j$ such that for all $A_i$ we have
\begin{equation*}
\sum_{L_1(x) =WL_2(y^2) + L_3(z)} 1_{A_i}(x_1)\dotsm 1_{A_i}(x_{s_1})1_{C_j}(y_1)\dotsm 1_{C_j}(y_{s_2}) 1_{C_j}(z_1)\dotsm 1_{C_j}(z_{s_3})\\ \geq \eta 
 N^{s_1 + \frac12\brac{s_2 + s_3} -1}W^{-\frac12\brac{s_2 + s_3}}.
\end{equation*}
\end{theorem}

We prove Theorem \ref{linear density} using Fourier analysis and the arithmetic regularity lemma. To state the regularity lemma we require the following.
\begin{definition}[Lipschitz constant on $\T^d$]
We say that $F : \T^d \to \C$ is $M$-Lipschitz if for any $\alpha, \beta \in \T^d$ we have
$$
|F(\alpha)-F(\beta)| \leq M \min_{1\leq i \leq d} \norm{\alpha_i - \beta_i}_\T.
$$
\end{definition}

\begin{lemma}[Arithmetic regularity]\label{arithmetic regularity}
Let $\eps > 0$ and let $\mathcal{F} : \N \to \N$.  For any functions $f_i : [N] \to [0, 1]$ with $i = 1, \dots, r$ there exists  $M \ll_{\eps, \mathcal{F},r} 1$ and decompositions 
$$
f_i = f_i^\mathrm{str} + f_i^\mathrm{sml} + f_i^\mathrm{unf} \qquad (1 \leq i \leq r),
$$
with the following properties.
\begin{enumerate}
\item[\emph{\textbf{(Str).}}] There exist $d\leq M$ and $\theta \in \T^d$, such that for each $i$ there is an $M$-Lipschitz function $F_i : \T^d \to [0,1]$ with 
$f_i^\mathrm{str}(x) = F_i( \theta x)$ for all $x \in [N]$.
\item[\emph{\textbf{(Sml).}}] $f_i^\mathrm{sml} : [N] \to [-1,1]$ with $$\sum_{x \in [N]} f_i^\mathrm{sml}(x) = 0, \quad \norm{f_i^\mathrm{sml}}_1 \leq \eps \norm{1_{[N]}}_1 \quad \text{and} \quad f_i^\mathrm{str} + f_i^\mathrm{sml} \geq 0.$$
\item[\emph{\textbf{(Unf).}}]  $f_i^\mathrm{unf} : [N] \to [-1,1]$ with $$ \sum_{x \in [N]} f_i^\mathrm{unf}(x) = 0 \quad\text{and} \quad\bignorm{\hat{f}_i^\mathrm{unf}}_{\infty} \leq \norm{\hat{1}_{[N]}}_{\infty}/\mathcal{F}(M).$$
\end{enumerate}
\end{lemma}

\begin{proof}
This can be proved by following the arguments of  \cite{GreenTaoArithmetic} or \cite{TaoHigher}. \end{proof}
To prove Theorem \ref{linear density} we apply the regularity lemma to decompose each $1_{A_i}$ into a structured, small and uniform part. We eventually show that the small and uniform parts do not contribute substantially to our count of solutions. It is therefore necessary to show that the structured part has a large contribution.
\begin{lemma}[Structured counting lemma]
Let $L_1, L_2, L_3$ be linear forms, each in $s_i$ variables. % and let $P \in \Q[y_1, \dots, y_t]$ be a quadratic polynomial with no constant term and no cross terms\footnote{So the only monomials appearing take the form $y_i^k$ for some $k = 1, 2$.}.
Given an $M$-Lipschitz function $F : \T^d \to [0,1]$ and $\theta \in \T^d$, define $f : \Z \to [0,1]$ by
$$
f(x) :=\begin{cases} F(\theta x), & x \in [N];\\
				0, & x \notin [N]. \end{cases}
$$ 
For fixed $0 < \eta \leq 1/2$, define the Bohr set
$$
B_1:= \set{x \in [\eta N] : \norm{\theta_i x}_\T \leq \eta \text{ for all } i}.
$$
For integers $c,W \geq 1$ let $B_2'$ denote a subset of the quadratic Bohr set
$$
B_2 :=\set{x \in [\eta (N/W)^{1/2}] : c \mid x \text{ and }\norm{\theta_i Wx^2/c}_\T, \bignorm{\theta_i x/c}_\T \leq \eta \text{ for all } i}.
$$
Then 
\begin{multline*}
\sum_{x\in [N]}\sum_{\substack{ d_i \in B_1 \\ y_j, z_k \in  B_2'}} f(x) f(x+cd_1)\dotsm f(x+cd_{s_1}) f\brac{x +L_1(d)+\tfrac{WL_2(y^2) + L_3(z)}{c}}\\ \geq |B_1|^{s_1}|B_2'|^{s_2+s_3}\sum_{x\in [N]} f(x)^{s_1+1} - O_{c}(\eta M |B_1|^{s_1}|B_2'|^{s_2+s_3}N).
\end{multline*}
\end{lemma}

\begin{proof}
Suppressing dependence on $L_i$, there exists a constant $C \ll_{c} 1$ such that if $d_i \in B_1$, $y_j, z_k \in B_2$ and
\begin{equation}\label{restricted x range}
C\eta N \leq x \leq N - C\eta N
\end{equation}
then 
$$
x + cd_1,\ \dots,\ x+ cd_{s_1},\ x +L_1(d)+\tfrac{WL_2(y^2) + L_3(z)}{c} \in [N].
$$
Restricting our summation over $x$ to \eqref{restricted x range} introduces an error of $O_{c}(\eta |B_1|^{s_1}|B_2'|^{s_2+s_3}N)$.  On restricting in this manner, each term in our summation satisfies $f(x+cd_i) = f( x)+ O_c(\eta M)$ and
$$
 f\brac{x +L_1(d)+\tfrac{WL_2(y^2) + L_3(z)}{c}} = f(x) + O(\eta M).
$$
The result follows.
\end{proof}

\begin{lemma}[$L^1$-control]
Let $L_1, L_2, L_3$ be linear forms, each in $s_i$ variables and let $S_1, S_2$ be finite sets of integers, with every element of $S_2$ divisible by $c$.
For any 1-bounded functions $f_i : \Z \to \C$ with support in $[N]$ we have the estimate
\begin{multline*}
\abs{\sum_{x\in \Z}\sum_{\substack{ d_i \in S_1 \\ y_j, z_k \in  S_2}} f_0(x) f_1(x+cd_1)\dotsm f_{s_1}(x+cd_{s_1}) f_{-1}\brac{x +L_1(d)+\tfrac{WL_2(y^2) + L_3(z)}{c}}}\\ \leq \min_i\norm{f_i}_{L^1[N]}|S_1|^{s_1}|S_2|^{s_2+s_3} .
\end{multline*}
\end{lemma}

\begin{proof}
This is clear from the triangle inequality for $i = 0$ .  The same argument applies for other values of $i$ on changing  variables.
\end{proof}

Combining our structured count with $L^1$-control, we can prove a version of the structured counting lemma which allows for small perturbations in the $L^1$-norm.  The proof follows from the standard telescoping identity
\[
\prod_i g(x_i) - \prod_i f(x_i) = \sum_i(g(x_i)-f(x_i)) \prod_{j<i} g(x_j) \prod_{k > i} f(x_k)
\]

\begin{corollary}\label{structured corollary}
Let $L_1, L_2, L_3$ be linear forms, each in $s_i$ variables. Given an $M$-Lipschitz function $F : \T^d \to [0,1]$  and $\theta \in \T^d$, define $f : [N] \to [0,1]$ by
$$
f(x) := F(\theta x), \qquad (x \in [N]).
$$ 
For fixed $0 < \eps \leq 1/2$, define the Bohr set
$$
B_1:= \set{x \in [\eps N] : \norm{\theta_i x}_\T \leq \eps/M \text{ for all } i}.
$$
For integers $c,W \geq 1$ let $B_2'$ denote a subset of the quadratic Bohr set
$$
B_2 :=\set{x \in [\eps (N/W)^{1/2}] : c \mid x \text{ and }\norm{\theta_i Wx^2/c}_\T, \bignorm{\theta_i x/c}_\T \leq \eps/M \text{ for all } i}.
$$
Then for any function $g : [N] \to [-1, 1]$ with $\norm{f-g}_{L^1[N]} \leq \eps N$ we have
\begin{multline*}
\sum_{x\in [N]}\sum_{\substack{ d_i \in B_1 \\ y_j, z_k \in  B_2'}} g(x) g(x+cd_1)\dotsm g(x+cd_{s_1}) g\brac{x +L_1(d)+\tfrac{WL_2(y^2) + L_3(z)}{c}}\\ \geq |B_1|^{s_1}|B_2'|^{s_2+s_3}\sum_{x\in [N]} f(x)^{s_1+2} - O_{c}(\eps |B_1|^{s_1}|B_2'|^{s_2+s_3}N).
\end{multline*}
\end{corollary}

The uniform part of the decomposition afforded by the regularity lemma (Lemma \ref{arithmetic regularity}) has small Fourier coefficients. The next lemma shows that such functions make negligible contribution to the count in Theorem \ref{linear density}.

\begin{lemma}[Fourier control]\label{linear fourier control}
Let $L_1, L_2, L_3$ denote non-singular linear forms, each in $s_i$ variables with $s_1 \geq 2$ and $s_1+s_2 \geq 3$.  Let $W$ be a positive integer and suppose that $W = 1$ or $s_3 \geq 1$.  Then for any positive integer $N \geq W^3$, 1-bounded functions $f_1, \dots, f_{s_1} : [N] \to \C$
and set  $B \subset [(N/W)^{1/2}]$ we have the estimate
\begin{multline}\label{GVT}
\abs{\sum_{L_1(x) =WL_2(y^2) + L_3(z)} f_1(x_1)\dotsm f_{s_1}(x_{s_1})1_{B}(y_1)\dotsm 1_{B}(y_{s_2}) 1_{B}(z_1)\dotsm 1_{B}(z_{s_3})}\\ \ll N^{s_1 + \frac12(s_2 + s_3) -1}W^{-\frac12(s_2 + s_3)} \min_i\brac{\frac{\bignorm{\hat{f}_i}_\infty}{N}}^{1/3}.
\end{multline}
\end{lemma}

\begin{proof}
Write 
$$
\tilde{1}_B(\alpha) := \sum_{x \in B} e(\alpha x^2)
$$
and
\begin{equation}\label{form notation}
L_i(x) = c_1^{(i)} x_1 + \dots + c_{s_i}^{(i)} x_{s_i}.
\end{equation}
The orthogonality relations give the identity
\begin{multline}\label{fourier0}
\sum_{L_1(x) =WL_2(y^2) + L_3(z)} f_1(x_1)\dotsm f_{s_1}(x_{s_1})1_{B}(y_1)\dotsm 1_{B}(y_{s_2}) 1_B(z_1) \dotsm 1_B(z_{s_3})\\ = \int_\T \prod_i\hat{f}_i\brac{c_i^{(1)}\alpha} \prod_j \tilde1_B\brac{Wc_j^{(2)}\alpha}\prod_k\hat1_B\brac{c_k^{(3)}\alpha}\intd\alpha.
\end{multline}

Let us first suppose that $s_1 \geq 3$. Fix distinct integers $i, j, k \in [s_1]$.  In this case we may estimate all exponential sums involving $1_B$ trivially, then employ Parseval to bound  \eqref{fourier0} by
\begin{multline*}
N^{s_1-3}(N/W)^{\frac12(s_2+s_3)} \int_\T \abs{\hat{f}_i\brac{c_i^{(1)}\alpha}}\abs{\hat{f}_j\brac{c_j^{(1)}\alpha}}\abs{\hat{f}_k\brac{c_k^{(1)}\alpha}}\intd\alpha \\
\leq N^{s_1-3}(N/W)^{\frac12(s_2+s_3)}\bignorm{\hat{f}_i}_\infty \bignorm{\hat{f}_j}_2\bignorm{\hat{f}_k}_2 \leq N^{s_1-2}(N/W)^{\frac12(s_2+s_3)}\bignorm{\hat{f}_i}_\infty.
\end{multline*}

Henceforth we assume that $s_1 = 2$, and so $s_2 \geq 1$ (since $s_1 + s_2 \geq 3$).
Let us deal with the case in which $W = 1$ and $s_3 = 0$. Then the orthogonality relations show our count equals
\begin{equation*}
\sum_{L_1(x) =L_2(y^2)} f_1(x_1) f_{2}(x_2)1_{B}(y_1)\dotsm 1_{B}(y_{s_2}) \\ = \int_\T \hat{f}_1\brac{c_1^{(1)}\alpha}\hat{f}_2\brac{c_2^{(1)}\alpha} \prod_j \tilde1_B\brac{c_j^{(2)}\alpha}\intd\alpha.
\end{equation*}
By H\"older's inequality and the trivial bound on exponential sums, the Fourier integral is at most
$$
 N^{\frac12(s_2-1)}\bignorm{\hat f_{1}}_\infty^{1/3}\bignorm{\hat f_{1}}_2^{2/3} \bignorm{\hat f_{2}}_2\bignorm{\tilde 1_B}_6.
$$
By Parseval $\bignorm{\hat f_i}_2 \leq N^{1/2}$, and by (say) Bourgain's restriction estimate \cite{BourgainLambda} we have $\bignorm{\tilde 1_B}_6 \ll N^{1/3}$ (more elementary proofs exist for the latter). The estimate \eqref{GVT} now follows in this case.

Next let us deal with the case that $W$ is arbitrary, in which case we may assume that $s_3 \geq 1$, in addition to our assumptions that $s_1 = 2$ and $s_2 \geq 1$.   
As above, H\"older's inequality allows us to bound \eqref{fourier0} by
$$
 (N/W)^{\frac12(s_2+s_3-2)}\bignorm{\hat f_{1}}_\infty^{1/3}\bignorm{\hat f_{1}}_2^{2/3} \bignorm{\hat f_{2}}_2\bignorm{\tilde 1_B(c_1^{(2)}W\alpha)\hat 1_B(c_1^{(3)}\alpha)}_6.
$$
 Hence it suffices to prove that for non-zero integers $c$ and $c'$ we have the estimate
\begin{equation}\label{l6 estimate}
\bignorm{\tilde 1_B(cW\alpha)\hat 1_B(c'\alpha)}_6 \ll_{c, c'} N^{5/6}W^{-1}.
\end{equation}

By orthogonality, the sixth power of the norm in \eqref{l6 estimate} is bounded above by the number of solutions to the equation
\begin{equation}\label{6th moment equation}
cW(y_1^2 + y_2^2+y_3^2 - y_4^2 - y_5^2 - y_6^2)= c'(z_1 + z_2 + z_3 - z_4 - z_5 - z_6), \quad \brac{y_i, z_j \in [(N/W)^{1/2}]}.
\end{equation}
Suppose that both sides of \eqref{6th moment equation} are equal to $W m$ for some $m \in \Z$.  Then the size constraints on the right-hand side 
 force $|m| \ll_{c'} (N/W)^{1/2}W^{-1}$.  Hence there are at most $O_{c'}((N/W)^{1/2}W^{-1} +1)$ choices for $m$, and given this choice there are at most $(N/W)^{\frac52}$ choices for $(z_1, \dots, z_6)$. Furthermore, by orthogonality the number of choices for $(y_1, \dots, y_6)$ is at most
$$
\int_\T \abs{\tilde{1}_{[(N/W)^{1/2}]}(c\alpha)}^6e(\alpha m) \intd\alpha \ll \int_\T \abs{\tilde{1}_{[(N/W)^{1/2}]}(\alpha)}^6\intd\alpha \ll (N/W)^2,
$$
the latter following from (say) Bourgain's restriction estimate \cite{BourgainLambda} (again, more elementary proofs exist).  The required estimate \eqref{l6 estimate} follows.
\end{proof}

\begin{lemma}[Quadratic Bohr set bound]\label{bohr set bound}
Let $0 < \eta \leq 1/2$ and $\alpha, \beta \in \T^d$.  Then for any positive integer $N $, either $N \ll_{\eta, d} 1$ or
$$
\hash\set{x \in [N] : \norm{\alpha_i x^2}_\T, \bignorm{\beta_i x}_\T \leq \eta \text{ for all } i} \gg_{\eta, d}  N.
$$
\end{lemma}

\begin{proof}
This follows from Tao \cite[Ex.1.1.23]{TaoHigher}.
\end{proof}

We now begin our proof of Theorem \ref{linear density} in earnest.   Some of our summations become cleaner if we view functions $f : [N] \to \C$ as functions $f: \Z \to \C$ which are equal to zero outside of $[N]$. We first prove Theorem \ref{linear density} under the assumption that we have a colouring $C_1\cup \dots \cup C_r = [(N/W)^{1/2}]$ of the full interval. We deduce the stated version subsequently.

We apply Lemma \ref{arithmetic regularity} to the indicator functions of the sets $A_i$ for $i =1, \dots, r$.  The precise values of $\eps$ and $\mathcal{F}$ in our use of this lemma are to be determined.  In this way we obtain $M \ll_{\eps, \mathcal{F},r} 1$ and decompositions
$$
1_{A_i} = f_i^\mathrm{str} + f_i^\mathrm{sml} + f_i^\mathrm{unf} \qquad (1 \leq i \leq r)
$$
which satisfy the conclusions of the arithmetic regularity lemma.  In particular, there exists $d\leq M$ and $\theta \in \T^d$, such that for each $i$ there is an $M$-Lipschitz function $F_i : \T^d \to [0,1]$ with 
$f_i^\mathrm{str}(x) = F_i( \theta x)$ for all $x \in [N]$. Since $f_i^\mathrm{str}$ has the same mean as $1_{A_i}$, we have
$$
\sum_{x \in [N]}f_i^\mathrm{str}(x)^{s_1} \geq \recip{N^{s_1 -1}}\brac{\sum_{x \in [N]}f_i^\mathrm{str}(x)}^{s_1} \geq \delta^{s_1} N.
$$
Write 
$$
L_1(x) = c_1 x_1 + \dots + c_{s_1} x_{s_1},
$$
and set
$$
\tilde{L}_1(d_3, \dots, d_{s_1}) = -\brac{c_3d_3 + \dots + c_{s_1}d_{s_1}}.
$$
Taking $B_1$ and $B_2$ as in Corollary \ref{structured corollary}, with $c:= c_1$, we deduce that for $g_i := f_i^\mathrm{str} + f_i^\mathrm{sml}$ and $B_2'\subset B_2$ we have
\begin{multline*}
\sum_{x\in [N]}\sum_{\substack{ d_i \in B_1 \\ y_j, z_k \in  B_2'}} g_i(x) g_i(x+c_1d_3)\dotsm g_i(x+c_1d_{s_1}) g_i\brac{x +\tilde{L}_1(d)+\tfrac{WL_2(y^2) + L_3(z)}{c_1}}\\ \geq \delta^{s_1}|B_1|^{s_1-2}|B_2'|^{s_2+s_3}N - O(\eps |B_1|^{s_1-2}|B_2'|^{s_2+s_3}N).
\end{multline*}
Hence we may take $\eps$ satisfying 
$
\eps^{-1} \ll \delta^{-s_1}
$
and ensure that
\begin{equation*}
\sum_{x\in [N]}\sum_{\substack{ d_i \in B_1 \\ y_j, z_k \in  B_2'}} g_i(x) g_i(x+c_1d_1)\dotsm g_i(x+c_1d_{s_1}) g_i\brac{x +\tilde{L}_1(d)+\tfrac{WL_2(y^2) + L_3(z)}{c_1}}\\ \gg \delta^{s_1}|B_1|^{s_1-2}|B_2'|^{s_2+s_3}N .
\end{equation*}

By the pigeon-hole principle, there exists a colour class $C_j$ satisfying \begin{equation}\label{PHP lower bound}
|C_j \cap B_2| \geq |B_2|/r.
\end{equation} 
 We observe that $C_j$ depends only on the Bohr set $B_2$, and not on the Lipschitz function $F_i$. On setting $B_2' := C_j \cap B_2$ and employing Lemma \ref{bohr set bound}, we deduce that
\begin{equation*}
\sum_{x\in [N]}\sum_{\substack{ d_i \in B_1 \\ y_j, z_k \in  B_2'}} g_i(x) g_i(x+c_1d_1)\dotsm g_i(x+c_1d_{s_1}) g_i\brac{x +\tilde{L}_1(d)+\tfrac{WL_2(y^2) + L_3(z)}{c_1}}\\ \gg_{\delta, r, M} N^{s_1-1}(N/W)^{\frac12(s_2+s_3)} .
\end{equation*}

Notice that under the assumption that $c_1 \mid y_j$ and $c_1 \mid z_k$, we obtain an integer solution to the equation $L_1(x) = WL_2(y^2)+ L_3(z)$ on setting
\begin{equation*}
x_1 := x + \tilde{L}_1(d) + \tfrac{WL_2(y^2) + L_3(z)}{c_1},\qquad
x_2 := x,\qquad x_3 := x + c_1d_3,\qquad \dots\ ,\qquad  x_{s_1}:= x + c_1 d_{s_1}.
\end{equation*}
Using the non-negativity of $g_i := f_i^\mathrm{str} + f_i^\mathrm{sml}$, we deduce that for $C = C_j$ satisfying \eqref{PHP lower bound} we have
\begin{equation}\label{g_i lower bound}
\sum_{\substack{L_1(x) = WL_2(y^2) + L_3(z)\\ y_j, z_k \in C\cap [(N/W)^{1/2}]}} g_i(x_1) \dotsm g_i(x_{s_1})  \gg_{\delta, r, M} N^{s_1-1}(N/W)^{\frac12(s_2+s_3)} .
\end{equation}
Employing Lemma \ref{linear fourier control} and a telescoping identity then gives
\begin{multline*}
\sum_{\substack{L_1(x) = WL_2(y^2) + L_3(z)\\ y_j, z_k \in C\cap [(N/W)^{1/2}]}} 1_{A_i}(x_1) \dotsm 1_{A_i}(x_{s_1}) = \sum_{\substack{L_1(x) = WL_2(y^2) + L_3(z)\\ y_j, z_k \in C\cap [(N/W)^{1/2}]}} g_i(x_1) \dotsm g_i(x_{s_1}) \\ + O\brac{ N^{s_1-1}(N/W)^{\frac12(s_2+s_3)} \mathcal{F}(M)^{-1/3}}
\end{multline*}
Hence taking $\mathcal{F}(M)$ sufficiently large in terms of the implicit constant in \eqref{g_i lower bound}, we conclude that for the colour class $C = C_j$ satisfying \eqref{PHP lower bound} we have
\begin{equation}\label{full colour lower bound}
\sum_{\substack{L_1(x) = WL_2(y^2) + L_3(z)\\ y_j, z_k \in C\cap [(N/W)^{1/2}]}} 1_{A_i}(x_1) \dotsm 1_{A_i}(x_{s_1}) \gg_{\delta, r} N^{s_1-1}(N/W)^{\frac12(s_2+s_3)}.
\end{equation}

This completes the proof of Theorem \ref{linear density} under the assumption that our colouring is of the full interval $C_1 \cup \dots \cup C_r = [(N/W)^{1/2}]$.  Notice that if $s_2+s_3 = 0$ then this vacuously implies the stated version of Theorem \ref{linear density}.  Let us therefore assume that $s_2 + s_3 > 0$.  Let $\eta$ denote the implicit constant in \eqref{full colour lower bound} divided through by $2(s_2+s_3)$.  Since the inverse image of the linear form $L_1$ has size at most $N^{s_1-1}$ in $[N]^{s_1}$, the number of solutions to the equation $L_1(x) = WL_2(y^2) + L_3(z)$ with $x_i \in [N]$, $y_j, z_k \in [(N/W)^{1/2}]$ and either $y_j \leq \eta (N/W)^{1/2}$ for some $j$ or $z_k \leq \eta (N/W)^{1/2}$ for some $k$ is at most
\begin{equation*}
 (s_2+s_3)\eta N^{s_1-1}(N/W)^{\frac12(s_2+s_3)}.
\end{equation*}
It follows that the bound \eqref{full colour lower bound} remains valid under the assumption that $$C_1 \cup \dots \cup C_r = [\eta (N/W)^{1/2}, (N/W)^{1/2}],$$ as required to prove Theorem \ref{linear density} in full generality.

\section{An abstract restriction estimate}\label{abstract restriction section}

To prove the quadratic counting theorem (Theorem \ref{quadratic form satisfies rado}) we would like to prove an analogue of the Fourier control lemma (Lemma \ref{linear fourier control}), which was key to our proof of the linear counting theorem (Theorem \ref{linear form satisfies rado}). The main ingredients in our proof of the Fourier control lemma were H\"older's inequality and estimates for the $L^p$-norm of certain exponential sums. In order to prove an analogous result for our quadratic counting theorem (see Lemma \ref{fourier control}) we require four distinct $L^p$-estimates, each of which involves the product of two distinct exponential sums (see Lemma \ref{mixed restriction}). We term these \emph{mixed restriction estimates}. To avoid repetition, we begin by proving an abstract restriction estimate, then verify that the four exponential sums of relevance satisfy the hypotheses of this theorem.

In the following $[-N, N]$ denotes an interval of integers.

\begin{definition}[Major arc hypothesis]\label{major hyp}
We say that $\nu: [-N,N] \to [0, \infty)$ satisfes a \emph{major arc hypothesis with constant $K$} if for all $1\leq a \leq q \leq Q$ with $\hcf(a, q) = 1$ and $\bignorm{\alpha - \tfrac{a}{q}}_\T\leq Q/N$ we have
\begin{equation}\label{major arc bound}
\frac{| \hat\nu(\alpha)|}{\norm{\nu}_1} \leq K q^{-1}  \max\set{1, \bignorm{\alpha - \tfrac{a}{q}}_\T N}^{-1} +\ Q^{O(1)}N^{-\Omega(1)}.
\end{equation}
If $K = O(1)$ then we simply say that $\nu$ satisfes a \emph{major arc hypothesis}.
\end{definition}

\begin{definition}[Minor arc hypothesis]\label{minor hyp} 
We say that $\nu : [-N,N] \to [0, \infty)$ satisfies a \emph{minor arc hypothesis} if for any $\delta > 0$ we have the implication
\[
|\hat\nu(\alpha)| \geq \delta \norm{\nu}_1 \quad \implies \quad \sqbrac{\exists q \ll \delta^{-O(1)} \text{ such that } \norm{q\alpha}_\T \ll\delta^{-O(1)}/N}.
\]
Here it is implicitly understood that $q$ is a positive integer.
\end{definition}

\begin{definition}[Hua-type hypothesis]\label{hua hyp}
We say that $\nu : [-N,N] \to [0, \infty)$ satisfies the \emph{Hua-type hypothesis with exponent $\eps$} if we have the bound
\begin{equation*}
\norm{\nu}_2 \leq \norm{\nu}_1 N^{\eps -1}.
\end{equation*}
\end{definition}

\begin{theorem}[Restriction estimate]\label{abstract restriction}  
Let $\nu : [-N,N] \to [0, \infty)$ satisfy the major and minor arc hypotheses, the major arc hypothesis with constant $K$. Given $p > 2$ there exists\footnote{The value of $\eps$ depends on the implicit constants in the major and minor arc hypotheses. However, these are absolute constants in our applications.} $\eps = \Omega_p(1)$ such that if $\nu$ satisfies the Hua-type hypothesis with exponent $\eps$, then for any 1-bounded function $\phi : [-N,N] \to \C$ we have
$$
\int_\T  \abs{\widehat{\phi\nu}(\alpha)}^p \intd\alpha \ll_{K,p}     \norm{\nu}_1^{p}N^{-1}.
$$
\end{theorem}

Our proof of Theorem \ref{abstract restriction} follows Bourgain's distributional approach \cite{BourgainLambda}, which has a nice exposition due to Henriot and Hughes \cite{HenriotHughesRestriction}.

\begin{lemma}[Distributional estimate]\label{quad lin-quad level set estimate}
Let $\nu : [-N,N] \to [0, \infty)$ satisfy the major and minor arc hypotheses, the major arc hypothesis with constant $K$.
Given\footnote{Our assumption that $\delta \leq 1/2$ is a convenience which allows us to replace bounds of the form $O(\delta^{-O(1)})$ with  $\delta^{-O(1)}$.} $0< \delta \leq 1/2$ and a 1-bounded function $\phi : [-N,N] \to \C$,  define
\begin{equation}\label{quad level set}
E_\delta( \phi, \nu) := \set{\alpha \in \T : \abs{\widehat{\phi\nu}(\alpha)}> \delta\norm{\nu}_1 }.
\end{equation}
Then for any  $\eps > 0$, either\footnote{The careful reader will observe that the implicit constants in our conclusion depend on the implicit constants in our major/minor arc hypotheses.} $N \leq \delta^{-O_{\eps}(1)}$ or 
$$
\meas\bigbrac{ E_\delta( \phi, \nu)}\ll_{K, \eps } N^{-1} \delta^{-2-\eps}.
$$
\end{lemma}

\begin{proof}[Proof of Theorem \ref{abstract restriction} given Lemma \ref{quad lin-quad level set estimate}]
Let $E_\delta$ be as in \eqref{quad level set} with $0 < \delta\leq 1/2$. By Lemma \ref{quad lin-quad level set estimate} with $\eps = \frac{p}{2}-1$, either $N \leq \delta^{-O_{ p}(1)}$ or 
$$
\meas(E_\delta) \ll_{K,p} N^{-1} \delta^{-1-\frac{p}{2}} .
$$
It follows that there exists $\Delta \leq N^{-\Omega_{p}(1)}$ such that for any $\delta \in (\Delta,1/2]$ we have
\begin{equation}\label{quad delta level set}
\meas(E_{\delta}) \ll_{K, p}N^{-1}\delta^{-1-\frac{p}{2}}.
\end{equation}
By dyadic decomposition
\begin{equation*}
\int_\T  \abs{\widehat{\phi\nu}(\alpha)}^p \intd\alpha \leq \int_{\T\setminus E_{\Delta}}  \abs{\widehat{\phi\nu}(\alpha)}^p \intd\alpha  + \sum_{1\leq j < \log_2(1/\Delta)} \int_{E_{2^{-j}}\setminus E_{2^{1-j}}} \abs{\widehat{\phi\nu}(\alpha)}^p \intd\alpha.
\end{equation*}

Since $\Delta \leq N^{-\Omega_{p}(1)}$, we can take $\eps = \eps(p)$ sufficiently small in our Hua-type hypothesis (Definition \ref{hua hyp}) to deduce that 
\begin{align*}
\int_{\T\setminus E_\Delta}  \abs{\widehat{\phi\nu}(\alpha)}^p \intd\alpha \leq(\Delta \norm{\nu_1}_1)^{p-2}  \int_\T  \abs{\widehat{\phi\nu}(\alpha)}^2 \intd\alpha
&\leq  \Delta^{p-2} \norm{\nu}_1^p N^{\eps-1} \leq \norm{\nu}_1^p N^{-1}.
\end{align*}

By \eqref{quad delta level set} we have
 \begin{align*}
\sum_{1\leq j < \log_2(1/\Delta)} \int_{E_{2^{-j}}\setminus E_{2^{1-j}}} \abs{\widehat{\phi\nu}(\alpha)}^p \intd\alpha 
& \leq  \norm{\nu}_1^{p} \sum_{1\leq j < \log_2(1/\Delta)} 2^{(1-j)p}\meas(E_{2^{-j}}) \\ 
& \ll_{K, p} \norm{\nu}_1^{p}N^{-1} \sum_{j=1}^\infty 2^{(1-j)p}2^{j(1+\frac{p}{2})}.
 \end{align*}
 The latter sum converges to an absolute constant of order $O_p(1)$ since $p > 2$.
\end{proof}

Our proof of Lemma \ref{quad lin-quad level set estimate} utilises the following divisor bound.
\begin{lemma}\label{bourgain divisor bound} Let
\begin{equation}\label{partial divisor}
d(n, Q) := \sum_{\substack{1\leq  q \leq Q\\ q \mid n}} 1.
\end{equation}
Then for any integer $B \geq 1$ and any real $X \geq 1$ we have
$$
\sum_{|n| \leq X}  d(n, Q)^B \ll_{\eps, B} Q^B + Q^\eps X.
$$
\end{lemma}

\begin{proof}
We follow Bourgain \cite[p.307]{BourgainLambda}:
\begin{multline*}
\sum_{|n| \leq X}  d(n, Q)^B  = \sum_{1\leq q_1, \dots, q_B \leq Q} \sum_{\substack{|n|\leq X\\ q_i \mid n}}1 \ll \sum_{1 \leq q_1, \dots, q_B \leq Q}  \brac{1+\frac{X}{[q_1, \dots, q_B]} }
\leq\\ Q^B + X\sum_{1 \leq q \leq Q^B} \frac{d(q)^B}{q} \ll_{\eps, B} Q^B + X\sum_{1 \leq q \leq Q^B} \frac{1}{q^{1-\frac{\eps}{B}}} \ll_{\eps, B} Q^B + Q^\eps X.\qedhere
\end{multline*}
\end{proof}

\begin{proof}[Proof of Lemma \ref{quad lin-quad level set estimate}]
Write $E_\delta := E_\delta(\phi, \nu)$ and
$$
 \epsilon(\alpha) := \begin{cases}\frac{\widehat{\phi\nu}(\alpha)}{\abs{\widehat{\phi\nu}(\alpha)}}, & \text{if } \widehat{\phi\nu}(\alpha)\neq 0;\\ 0, & \text{otherwise}.\end{cases}
$$
Then 
\begin{multline*}
\delta \norm{\nu}_1 \meas(E_\delta) \leq \int_{E_\delta} \abs{\widehat{\phi\nu}(\alpha)}\intd\alpha = \int_{E_\delta} \widehat{\phi\nu}(\alpha)\overline{\epsilon(\alpha)} \intd\alpha\\
= \sum_{ x} \phi(x)\sqrt{\nu(x)}\sqrt{\nu(x)}\int_{E_\delta} e(\alpha x ) \overline{\epsilon(\alpha)}\intd\alpha\\
 \leq  \norm{\nu}_1^{1/2}\brac{\sum_{ x} \nu(x)\abs{\int_{E_\delta} e(\alpha x)\overline{\epsilon(\alpha)}\intd\alpha}^2}^{1/2}.
\end{multline*}
Expanding absolute values, then using linearity of integration and the triangle inequality, we have
\begin{equation*}
\sum_{ x }\nu(x)\abs{\int_{E_\delta} e(\alpha x)\overline{\epsilon(\alpha)}\intd\alpha}^2
\leq \int_{E_\delta}\int_{E_\delta} |\hat\nu(\alpha_1 -\alpha_2)|\intd\alpha_1\intd\alpha_2.
\end{equation*}
Hence we deduce the Tomas-Stein inequality
\begin{equation*}
\delta^2 \norm{\nu}_1 \meas(E_\delta)^2 \leq  \int_{E_\delta}\int_{E_\delta} |\hat\nu(\alpha_1 -\alpha_2)|\intd\alpha_1\intd\alpha_2.
\end{equation*}

Consider the Fej\'er kernel
$$
F_{N}(\alpha) := N^{-1} \abs{\hat{1}_{[N]}(\alpha)}^2 =\sum_n \brac{1-\tfrac{|n|}{N}}_+ e(\alpha n),
$$
a trigonometric polynomial of degree $N-1$ which is also a probability measure on $\T$. 
Set
$$
\psi_N(\alpha) := \brac{1 + e(\alpha N) + e(-\alpha N)}F_{N}(\alpha).%,
$$
%a trigonometric polynomial of degree $2N-1$.  
For $ |n| \leq N$ one can check that
$$
\hat{\psi}_N(n) = \brac{1-\tfrac{|n|}{N}}_+ + \brac{1-\tfrac{|n-N|}{N}}_++ \brac{1-\tfrac{|n+N|}{N}}_+ = 1.
$$
One can also check that if $n \in \Z\setminus [-N,N]$ then
$$
\int_\T \hat{\nu}(\alpha) e(-\alpha n) \intd\alpha = 0.
$$
Therefore the Fourier coefficients of $\alpha \mapsto \hat{\nu}(\alpha)$ agree with those of the convolution
$$
\hat{\nu} *\psi_{N}: \alpha \mapsto \int_{\T} \hat{\nu}(\alpha-\beta)\psi_{N}(\beta) \intd\beta.
$$
By Fej\'er's theorem \cite[Theorem 3.1]{KatznelsonIntroduction} these functions must be identical and we deduce that
\begin{multline*}
\delta^2 \norm{\nu}_1 \meas(E_\delta)^2 
\leq \int_{\T^3}  |\hat\nu(\alpha_1-\alpha_2 -\beta)\psi_{N}(\beta)|1_{E_\delta}(\alpha_1)1_{E_\delta}(\alpha_2)\intd\alpha_1\intd\alpha_2\intd\beta\\
\ll \int_{\T^3}  |\hat\nu(\alpha_1-\alpha_2 -\beta)|F_{N}(\beta)1_{E_\delta}(\alpha_1)1_{E_\delta}(\alpha_2)\intd\alpha_1\intd\alpha_2\intd\beta.
\end{multline*}

Let $Q \geq 1$ (to be determined) and write
\begin{equation}\label{quad major arcs}
\mathfrak{M} := \bigcup_{\substack{1\leq a \leq q \leq Q\\ \hcf(a,q) = 1}} \set{\alpha \in \T : \bignorm{\alpha - \tfrac{a}{q}}_\T\leq Q/N}.%, \quad  \mathfrak{m} := \T\setminus\mathfrak{M}.
\end{equation}
Our minor arc hypothesis (Definition \ref{minor hyp}) shows that if $\alpha \in  \T\setminus\mathfrak{M}$ then 
$
\abs{\hat\nu(\alpha)} \ll  Q^{-\Omega(1)} \norm{\nu}_1.
$
Hence 
\begin{equation*}
\int_{\alpha_1 - \alpha_2 - \beta \in \T\setminus\mathfrak{M}}  |\hat\nu(\alpha_1-\alpha_2 -\beta)|F_{N}(\beta)1_{E_\delta}(\alpha_1)1_{E_\delta}(\alpha_2)\intd\alpha_1\intd\alpha_2\intd\beta\\ \ll Q^{-\Omega(1)}\norm{\nu}_1\meas(E_\delta)^2 .
\end{equation*}
 It follows that either $Q \leq \delta^{-O(1)} $ or
\begin{equation*}
\delta^2 \norm{\nu}_1 \meas(E_\delta)^2\\
\ll \int_{\alpha_1 - \alpha_2 - \beta \in \mathfrak{M}}  |\hat\nu(\alpha_1-\alpha_2 -\beta)|F_{N}(\beta)1_{E_\delta}(\alpha_1)1_{E_\delta}(\alpha_2)\intd\alpha_1\intd\alpha_2\intd\beta.
\end{equation*}
Letting $C = O(1)$ denote a sufficiently large absolute constant, set
\begin{equation}\label{Q defn}
Q: = \delta^{-C}  \leq \delta^{-O(1)}.
\end{equation}
Then, for any $p \geq 1$,  H\"older's inequality yields
\begin{equation*}
\delta^{2p} \norm{\nu}_1^p \meas(E_\delta)^2\\
\ll \int_{\alpha_1 - \alpha_2 - \beta \in \mathfrak{M}}  |\hat\nu(\alpha_1-\alpha_2 -\beta)|^pF_{N}(\beta)1_{E_\delta}(\alpha_1)1_{E_\delta}(\alpha_2)\intd\alpha_1\intd\alpha_2\intd\beta.
\end{equation*}
Our major arc hypothesis (Definition \ref{major hyp}) implies that for $p \in [1, 2]$ we have the bound 
\begin{equation*}
\frac{1_{\mathfrak{M}}(\alpha) \abs{\hat\nu(\alpha)}^p}{\norm{\nu}_1^p} \ll_K \sum_{1\leq a\leq q \leq Q} q^{-1} \max\bigset{1, \bignorm{\alpha -\tfrac aq}_\T N}^{-p} +\ Q^{O(1)}N^{-\Omega(1)}.
\end{equation*}
Hence for $p \in [1,2]$ we deduce that either $N \leq \delta^{-O(1)}$ or 
\begin{equation}\label{quad pre holder}
\delta^{2p}  \meas(E_\delta)^2  \ll_K \\ \sum_{1\leq a\leq q \leq Q} q^{-1} \int_{E_\delta\times E_\delta\times\T}  \max\set{1, \bignorm{\alpha_1-\alpha_2-\beta - \tfrac{a}{q}}_\T N}^{-p}F_{N}(\beta).%\intd\alpha_1\intd\alpha_2\intd\beta.
\end{equation}

Set
$$
\mu_{Q, p}(\alpha) := \sum_{1\leq a \leq q \leq Q}q^{-1}\max\set{1, \bignorm{\alpha-\tfrac{a}{q}}_\T N}^{-p}.
$$
Then we can re-write \eqref{quad pre holder} as 
\begin{equation}\label{quad post holder}
\delta^{2p} \meas(E_\delta)^2  \ll_K \int_{E_\delta}\int_{E_\delta} \mu_{Q, p}*F_{N}(\alpha_1-\alpha_2)\intd\alpha_1\intd\alpha_2.
\end{equation}
Using the following normalisation for inner products
$$
\ang{f, g} := \int_\T f(\alpha) \overline{g(\alpha)} \intd \alpha,
$$
the inequality \eqref{quad post holder} implies that
$$
\delta^{2p} \meas(E_\delta)^2 \ll_K \ang{\mu_{Q, p}*F_{N} , 1_{E_\delta}*1_{-E_\delta}}.
$$
By Parseval's theorem \cite[Theorem 5.5(d)]{KatznelsonIntroduction} and the convolution identity \cite[Theorem 1.7]{KatznelsonIntroduction} we have
\begin{equation*}
 \ang{\mu_{Q, p}*F_{N} , 1_{E_\delta}*1_{-E_\delta}} 
=\ang{\hat{\mu}_{Q,p}\hat F_{N}, \bigabs{\hat{1}_{E_\delta}}^2}
\leq \sum_{|n| < N}\abs{\hat{\mu}_{Q,p}(n)}  \bigabs{\hat{1}_{E_\delta}(n)}^2.
\end{equation*}
Hence any $p \in [1,2]$ yields the estimate  
\begin{equation*}
\delta^{2p} \meas(E_\delta)^2  \ll_K \sum_{|n| < N}\abs{\hat{\mu}_{Q,p}(n)}  \bigabs{\hat{1}_{E_\delta}(n)}^2.
\end{equation*}

Recalling the definition \eqref{partial divisor} of $d(n,Q)$, a change of variables shows that for any $p > 1$ we have 
\begin{equation}
\abs{\hat\mu_{Q, p}(n) }\leq d(n,Q)  \int_\T \max\set{1, \norm{\alpha}_\T N}^{-p} \intd\alpha \ll \frac{d(n,Q)}{(p-1)N}.
\end{equation}
Hence, for any $B \geq 1$ and $p \in (1, 2]$,  H\"older's inequality gives  
\begin{align*}
\delta^{2p} \meas(E_\delta)^2 & \ll_K \frac{1}{(p-1) N} \Biggbrac{\sum_{|n|< N} d(n, Q)^B}^{1/B}\brac{\sum_{n} \bigabs{\hat{1}_{E_\delta}(n)}^{2B/(B-1)}}^{1-\recip{B}}.% \\ & \ll_{B, p} K_1N^{-1} (Q^B + Q N)^{1/B}\meas(E_\delta)^{1 +\recip{B}}.
\end{align*}
Applying Parseval again, together with Lemma \ref{bourgain divisor bound}, we conclude that for any $p \in (1, 2]$ and any integer $B\geq 1$  we have 
\begin{align*}
\delta^{2p} \meas(E_\delta)^2 & \ll_{K,B} \frac{1}{(p-1) N}  (Q^B + Q N)^{1/B}\meas(E_\delta)^{1 +\recip{B}}.
\end{align*}

Set $B := 1 + \ceil{1/\eps}$ and $p := 1 + B^{-1}$. Recalling our choice \eqref{Q defn} of $Q$, either $Q^B \leq QN$ or $N \leq \delta^{-O_{\eps}(1)}$. In the former case we have
\[
\meas(E_\delta) \ll_{K,\eps} N^{-1} Q^{\frac{1}{B-1}} \delta^{-\frac{2pB}{B-1}} \leq N^{-1} \delta^{-2 - O(\eps)}.
\]
The result  follows on rescaling $\eps$ to absorb the $O(1)$ constant in the exponent.
\end{proof}

\section{Exponential sum estimates}\label{Fourier Pseudorandomness}

The purpose of this section is to prove various bounds on exponential sums which are needed to verify the hypotheses required for our mixed restriction estimates (see Lemma \ref{mixed restriction}).  Three out of four of these mixed restriction estimates involve the following majorant, which also plays a prominent role in \cite{BrowningPrendivilleTransference} and \cite{CLPRado}. 

\begin{definition}[Majorant for squares]\label{square majorant}
For fixed $\xi \in [W]$, with $W$ even, define $\nu = \nu_{W, \xi} :[N] \to [0, \infty)$ by
\begin{equation}\label{majorant}
\nu(n) := \begin{cases} Wx+\xi, & \text{if } n = \frac{(Wx + \xi)^2 - \xi^2}{2W} \text{ for some } x \in \Z;\\
0, & \text{otherwise}.\end{cases}
\end{equation}
\end{definition}

Before estimating the Fourier transform of $\nu$ we recall Weyl's inequality for squares.

\begin{lemma}[Weyl's inequality]\label{weyl inequality}
Let $I$ be an interval of at most $N$ integers and let $\alpha, \beta \in \T$. Suppose that
$$
\abs{\sum_{n \in I} e(\alpha  n^2+\beta n)} \geq \delta N.
$$
Then there exists $q \ll \delta^{-O(1)}$ such that $\norm{q\alpha}_\T \ll \delta^{-O(1)}/N^2$.
\end{lemma}

\begin{proof}
See  Green--Tao \cite[Lemma A.11]{GreenTaoQuadratic}.
\end{proof}

\begin{corollary}[Coarse minor arc estimate for $\nu$] \label{QuadraticMinor}  Let $W$ be an even positive integer, $\xi \in [W]$ and define $\nu = \nu_{W, \xi}$ as in \eqref{majorant}.  Suppose that 
$$
\bigabs{\hat{\nu}(\alpha)} \geq \delta N.
$$
Then either $N \ll W$ or there exists $1\leq a \leq q \ll \delta^{-O(1)}W $ such that $\hcf(a, q) =1$ and $\bignorm{\alpha-\tfrac{a}{q}}_\T\ll \delta^{-O(1)}/N$.
\end{corollary}

\begin{proof}
Summation by parts gives that
\[
\sum_{x < n \leq y} f(n)a_n = f(y)\sum_{x < n \leq y} a_n - \int_x^y f'(t)\brac{\sum_{x< n\leq t} a_n} \intd t
\]
Hence
\begin{multline*}
\hat{\nu}(\alpha) = \sum_{0 < \recip{2} Wx^2 + \xi x \leq N} (Wx+\xi)e\brac{\alpha\brac{\trecip{2} Wx^2 + \xi x}} \\
= \brac{\sqrt{2WN + \xi^2}} \sum_{0 < \recip{2} Wx^2 + \xi x \leq N} e\brac{\alpha\brac{\trecip{2} Wx^2 + \xi x}} 
- W\int_0^{\frac{\sqrt{2WN + \xi^2}-\xi}{W}}  \sum_{0< x \leq t} e\brac{\alpha\brac{\trecip{2} Wx^2 + \xi x}} \intd t.
\end{multline*}
Since $\xi \in [W]$, one can check that either $N \ll W$ or the interval $\{x : 0 < \recip{2} Wx^2 + \xi x \leq N\}$ has length of order $\asymp \sqrt{N/W}$. It therefore follows that if $|\hat{\nu}(\alpha)| \geq \delta N$ then there exists $t \ll \sqrt{N/W}$ such that
$$
\abs{\sum_{0< \recip{2} Wx^2 + \xi x \leq t} e\brac{\alpha\brac{\trecip{2} Wx^2 + \xi x}}} \gg \delta \sqrt{N/W}
$$
Applying Weyl's inequality, there exists $q_0 \ll \delta^{-O(1)}$ such that $\norm{q_0 \alpha\trecip{2} W}_\T \ll \delta^{-O(1)}W/N$.  Setting $q := \trecip{2} W q_0$ then yields the result.
\end{proof}

%The following is a straightforward adaptation of \cite[Lemma 5.1]{BrowningPrendivilleTransference}.

\begin{lemma} [Major arc asymptotic for $\nu$] \label{MajorArcAsymptotic}
Let $W$ be an even positive integer, $\xi \in [W]$ and define $\nu = \nu_{W, \xi}$ as in \eqref{majorant}.
Suppose that $\| q \alpha \|_\T = |q \alpha - a|$ for some $q,a \in \Z$ with $q > 0$. Then either $N \ll W$ or
\begin{equation*}
\hat \nu(\alpha) = \E_{r\in [q]} e_q\brac{a\brac{\trecip{2} Wr^2 + \xi r}} \int_0^N e\brac{\brac{\alpha - \tfrac{a}{q}}t}\intd t\\ + O\brac{\brac{\sqrt{WN}+Wq}\brac{q + \norm{q\alpha}_{\T} N}}.
\end{equation*}
\end{lemma}

\begin{proof}
Writing $\beta := \alpha - \tfrac{a}{q}$ and summing over congruence classes mod $q$, we have
\begin{equation}\label{progression split}
\hat\nu(\alpha) = \sum_{r=1}^q e_q\brac{a\brac{\trecip{2} Wr^2 + \xi r}} \sum_{\substack{0< \recip{2} Wx^2 + \xi x \leq N\\ x \equiv r \bmod q}} (Wx+\xi)e\brac{\beta\brac{\trecip{2} Wx^2 + \xi x}}.
\end{equation}
Comparing the inner sum with an integral, as in \cite[Ex11]{TaoAPNTNotes1}, we have
\begin{equation*}
\sum_{\substack{0< \recip{2} Wx^2 + \xi x \leq N\\ x \equiv r \bmod q}} (Wx+\xi)e\brac{\beta\brac{\trecip{2} Wx^2 + \xi x}} = \\
q^{-1}\int_0^N e\brac{\beta t}\intd t 
 + O\brac{\brac{\sqrt{WN}+Wq}\brac{1+ N|\beta|}}.
\end{equation*}
Substituting this into \eqref{progression split} gives the result.
\end{proof}

\begin{lemma}[Local Weyl estimate]\label{local estimate} For any integers $a, b, q$ with $q$ positive we have
$$
|\E_{r\in [q]} e_q\brac{a r^2 + b r} | \ll \hcf(a,q)^{1/2}q^{-1/2}.
$$
\end{lemma}

\begin{proof}
Let $q_0 := q/\hcf(2a, q)$. Squaring and Weyl differencing gives
\begin{multline*}
|\E_{r\in [q]} e_q\brac{a r^2 + b r} |^2 \leq \E_{h \in [q]} \abs{\E_{r\in [q]} e_q\brac{a 2hr}}\\ = q^{-1}\hash\set{h \in [q] : 2a h \equiv 0 \bmod q} = q^{-1}\hash\set{h \in [q] : h  \equiv 0 \bmod q_0} \\
= q^{-1} \hcf(2a, q) \ll q^{-1} \hcf(a, q).\qedhere
\end{multline*}
\end{proof}
 
\begin{lemma}\label{W local weyl}
Let $W$ be an even positive integer and let $\xi \in [W]$ with $\hcf(\xi, W) =1$. Then for any positive integers $a$ and $q$ with $\hcf(a, q) = 1$ we have
$$
\E_{r\in [q]} e_q\brac{a\brac{\trecip{2} Wr^2 + \xi r}} \ll \begin{cases} 0, & \text{if } \hcf(q, \trecip{2}W) > 1;\\
q^{-1/2}, & \text{otherwise}.\end{cases}
$$
\end{lemma}

\begin{proof}
Write $q=q_0q_1$ where $q_1 = \hcf(\trecip{2}W, q)$.   Writing $r = r_0 + q_0r_1$ where $r_0 \in [q_0]$ and $r_1 \in [q_1]$ we have
\begin{multline}\label{large local weyl}
\abs{\E_{r\in [q]} e_q\brac{a\brac{\trecip{2} Wr^2 + \xi r}}} 
=\abs{\E_{r_0\in [q_0]} e_q\brac{a\trecip{2} Wr_0^2} \E_{r_1\in [q_1]} e_q\brac{a\brac{\xi (r_0 + q_0r_1)}}}\\ \leq \abs{ \E_{r_1\in [q_1]} e_{q_1}\brac{a\xi r_1}} = 1_{q_1 \mid \xi} .
\end{multline}
The estimate now follows if $q_1 = \hcf(\trecip{2}W, q) > 1$, since in this case $q_1  \nmid \xi$ because $\hcf(\xi, W) = 1$. The case when $\hcf(\trecip{2}W, q) = 1$ follows from Lemma \ref{local estimate}.
\end{proof}

\begin{lemma}\label{sqroot local weyl}
Let $W$ be an even positive integer and let $\xi \in [W]$ with $\hcf(\xi, W) =1$. Then for any positive integers $a$ and $q$ we have
$$
\E_{r\in [q]} e_q\brac{a\brac{\trecip{2} Wr^2 + \xi r}} \ll \begin{cases} 0 & \text{if $\trecip{2}W$ and $\frac{q}{(a,q)}$ are not coprime;}\\ (a,q)^{1/2}q^{-1/2} & \text{otherwise.}\end{cases}
$$
\end{lemma}

\begin{proof}
Write $q=q_0q_1$ and $a = a_0q_1$ where $q_1 = \hcf(a, q)$.   Summing over residues mod $q_0$ we have
\begin{equation*}
\abs{\E_{r\in [q]} e_q\brac{a\brac{\trecip{2} Wr^2 + \xi r}}} 
=\abs{\E_{r\in [q_0]} e_{q_0}\brac{a_0\brac{\trecip{2} Wr^2 + \xi r}}}.
\end{equation*}
The result follows on applying  Lemma \ref{W local weyl}.
\end{proof}

\begin{lemma}[Refined minor arc estimate for $\nu$]\label{RefinedMinor}  Let $W$ be an even positive integer and $\xi \in [W]$ with $\hcf(\xi, W) =1$. Define $\nu = \nu_{W, \xi}$ as in \eqref{majorant}.  Suppose that 
$$
\bigabs{\hat{\nu}(\alpha)} \geq \delta N.
$$
Then either $N \ll W^{O(1)}$ or there exists $ q \ll \delta^{-O(1)} $ such that $\norm{q\alpha}_\T\ll \delta^{-O(1)}/N$. In particular, either $N\ll W^{O(1)}$ or $\nu$ satisfies the minor arc hypothesis (Definition \ref{minor hyp}).
\end{lemma}

\begin{proof}
Applying Corollay \ref{QuadraticMinor}, there exists $1 \leq a \leq q \ll W\delta^{-O(1)}$ for which $\hcf(a, q) = 1$ and $\bignorm{\alpha -\tfrac{a}{q}}_\T \ll \delta^{-O(1)}/N$.  By Lemma \ref{MajorArcAsymptotic}, either $N \ll (W/\delta)^{O(1)}$ or
$$
\abs{\E_{r\in [q]} e_q\brac{a\brac{\trecip{2} Wr^2 + \xi r}}} \gg \delta.
$$
Applying Lemma \ref{W local weyl}, we deduce that $q \ll \delta^{-2}$. Finally we note that $N \ll (W/\delta)^{O(1)}$ implies that either $N \ll W^{O(1)}$ or $N \ll \delta^{-O(1)}$, and the conclusion of the minor arc hypothesis is trivial if the latter holds.
\end{proof}

\begin{lemma}[Linear exponential estimates]\label{linear exp}
Let $I \subset \R$ be an interval and $\beta \in \R$. Then
\begin{equation*}%\label{linear major approx}
\int_I e(\beta t) \intd t  \ll \min\set{\meas(I), |\beta|^{-1}}.
\end{equation*}
If $\alpha \in \T$ then
\begin{equation*}%\label{linear major approx}
\sum_{x \in I} e(\alpha x)\ll \min\set{\meas(I)+1, \norm{\alpha}_\T^{-1}}.
\end{equation*}
Furthermore
\begin{equation*}%\label{linear major approx}
\sum_{x \in I} e(\beta x)- \int_I e(\beta t) \intd t\ll 1+|\beta|\meas(I).
\end{equation*}
\end{lemma}

\begin{proof}
The first estimate follows from integration, the second from summing the geometric series, the third  from  approximating a sum by an integral as in \cite[Ex11]{TaoAPNTNotes1}.
\end{proof}

\begin{lemma}[Fourier decay]\label{fourier decay}
Suppose that $W$ is divisible by $2\prod_{p \leq w} p$ and that $\hcf(\xi, W) = 1$.  Define $\nu = \nu_{W, \xi}$ as in \eqref{majorant}.  Then either $N \ll W^{O(1)}$ or 
$$
\norm{\hat{\nu}-\hat{1}_{[N]}}_\infty \ll w^{-1/2}N.
$$
\end{lemma}

\begin{proof}%[Proof of Lemma \ref{fourier decay}]
First suppose that $|\hat{\nu}(\alpha)-\hat{1}_{[N]}(\alpha)| \geq \delta N$. Then by the triangle inequality, either $|\hat{\nu}(\alpha)| \gg \delta N$ or $|\hat{1}_{[N]}(\alpha)| \gg \delta N$.  In the latter case, Lemma \ref{linear exp} gives that $\norm{\alpha}_\T \ll \delta^{-1}/N$.  We claim that a similar conclusion holds under the assumption that $|\hat{\nu}(\alpha)| \gg\delta N$.

To establish the claim we first repeat the argument of Lemma \ref{RefinedMinor} to conclude that either $N \ll (W/\delta)^{O(1)}$ or there exists  $1 \leq a \leq q \ll \delta^{-2}$ with $\hcf(a,q) = 1$  such that $\bignorm{\alpha -\tfrac{a}{q}}_\T \ll \delta^{-O(1)}/N$ and \begin{equation}\label{q largeness}
\abs{\E_{r\in [q]} e_q\brac{a\brac{\trecip{2} Wr^2 + \xi r}} \int_0^N e\brac{\brac{\alpha - \tfrac{a}{q}}t}\intd t} \gg \delta N.
\end{equation}
Applying Lemma \ref{W local weyl}, we deduce that $\hcf(q, \trecip{2}W) = 1$. Since we are assuming that $ \trecip{2}W$ is divisible by all primes $p \leq w$, we conclude that $q > w$ or $q = 1$.  If $q =1$ then we may bound the integral in \eqref{q largeness} using Lemma \ref{linear exp} to deduce that  $\norm{\alpha}_\T \ll \delta^{-1}/N$, as claimed.

We may therefore conclude that the assumption $|\hat{\nu}(\alpha)-\hat{1}_{[N]}(\alpha)| \geq \delta N$ implies that either $N \ll (W/\delta)^{O(1)}$, or $w \ll \delta^{-2}$ or
\begin{equation}\label{zero approx}
\norm{\alpha}_\T\ll \delta^{-1}/N.
\end{equation}  
Supposing that \eqref{zero approx} holds, if we substitute the approximations given by Lemma \ref{linear exp} and Lemma \ref{MajorArcAsymptotic} into the inequality $|\hat{\nu}(\alpha)-\hat{1}_{[N]}(\alpha)| \geq \delta N$, then again we deduce that $N \ll (W/\delta)^{O(1)}$. 

To summarise: if $|\hat{\nu}(\alpha)-\hat{1}_{[N]}(\alpha)| \geq \delta N$ then either $N \ll (W/\delta)^{O(1)}$ or $w \ll \delta^{-2}$.  The lemma is complete on taking $\delta := C w^{-1/2}$ for a sufficiently large absolute constant $C$, and on observing that $w \leq 2\prod_{p\leq w} p \leq W$ (by Bertrand's postulate, for instance).
\end{proof}

\begin{lemma}[Quadratic major arc asymptotic]\label{quad maj asymp}
Let $W$ be a positive integer, $\eta \in (0, 1/2]$ and define the interval
$$
I := \sqbrac{\eta(N/W)^{1/2}, (N/W)^{1/2}}.
$$
Suppose that $\| q \alpha \| = |q \alpha - a|$ for some $q,a \in \Z$ with $q > 0$. Then either $N \ll W^{O(1)}$ or
\begin{equation}\label{quad major}
\sum_{x \in I} e(\alpha Wx^2)  = W^{-1/2}\E_{r\in[q]} e_q\brac{aWr^2}\int_{\eta \sqrt{N}}^{\sqrt{N}} e\brac{\beta t^2}\intd t+ O\brac{q + \norm{q\alpha}_\T N}.
\end{equation}
\end{lemma}

\begin{proof}

Let $\alpha \in \mathfrak{M}(a,q)$ and let $\beta$ denote the least absolute real in the congruence class $ \alpha - \tfrac{a}{q} \pmod 1$. Summing over residues mod $q$, we have
\begin{equation}\label{S_2 progression split}
\sum_{x \in I} e(\alpha Wx^2) = \sum_{r=1}^q e_q\brac{aWr^2} \sum_{\substack{x \in I\\ x \equiv r \bmod q}} e\brac{\beta Wx^2}.
\end{equation}
Comparing the inner sum with an integral as in \cite[Ex11]{TaoAPNTNotes1} gives
\begin{equation*}
\sum_{\substack{x \in I\\ x \equiv r \bmod q}} e\brac{\beta Wx^2} =  q^{-1}W^{-1/2}\int_{\eta \sqrt{N}}^{\sqrt{N}} e\brac{\beta t^2}\intd t+ O\brac{1+|\beta|N}.
\end{equation*}
Substituting this into \eqref{S_2 progression split} gives \eqref{quad major}.
\end{proof}

\begin{lemma}[Quadratic exponential integral bound]\label{integral bound} For $\beta \in \R$ we have
$$
\abs{\int_{\eta N}^N e(\beta t^2) \intd t} \ll N\max\set{1, \eta|\beta| N^{2}}^{-1}.
$$					
%In particular, for $1< p < 2$ we have
%$$
% \int_\T |I_1(\beta)I_2(\beta)|^{p} d\beta  \ll \frac{1}{p-1} N^{2p-2}W^{-1}.
%$$
\end{lemma}

\begin{proof}
Let us show that for $\beta > 0$ we have
$$
\abs{\int_{\eta N}^N e(\beta t^2) \intd t} \ll \recip{\eta\beta N}.
$$
The claimed bound then follows on incorporating the trivial estimate of $N$, and utilising conjugation to deal with $\beta<0$. By a change of variables
$$
\int_{\eta N}^N e(\beta t^2) \intd t = \beta^{-1/2} \int_{\eta^2 \beta N^2}^{\beta N^2} \frac{e(v)}{2v^{1/2}} \intd v.
$$
Integrating by parts shows that for $0 < x \leq y$ we have
\[
\int_x^y \frac{e(t)}{t^{1/2}} \intd t \ll x^{-1/2}.\qedhere
\]
\end{proof}

\begin{lemma}[Major arc hypotheses]\label{major hyp veri}
Let $W_1$ and $W_2$ be $w$-smooth positive integers such that $W_1$ is divisible by $2 \prod_{p \leq w} p$. Given $\xi \in [W_1]$ with $\hcf(\xi, W_1) =1$, define $\nu = \nu_{W_1, \xi} : [N] \to [0, \infty)$ as in \eqref{majorant}.
Given $\eta \in (0, 1/2]$, define the interval
$$
I := \sqbrac{\eta(N/W_2)^{1/2}, (N/W_2)^{1/2}}.
$$
Fix non-zero integers $b_1, b_2 = O(1)$ and write $B:= |b_1|+|b_2|$. Consider the following four majorants,  mapping each $n \in [-BN,BN ]$ to one of
\begin{equation}\label{majorants2}
\sum_{b_1x +b_2y = n} \nu(x)\nu(y), \quad  \sum_{b_1x + b_2W_2y^2=n} \nu(x)1_I(y), \quad  \sum_{b_1x + b_2y=n}  \nu(x)1_I(y), \quad  \sum_{b_1W_2x^2 + b_2y=n} 1_I(x)1_I(y).
\end{equation}
Then either $N \ll (W_1W_2)^{O(1)}$ or all four majorants satisfy the major arc hypothesis (Definition \ref{major hyp}), the latter with constant $\eta^{-O(1)}$.
\end{lemma}

\begin{proof} Let $1\leq a \leq q \leq Q$ with $\hcf(a, q) = 1$ and $\bignorm{\alpha - \tfrac{a}{q}}_\T\leq Q/(BN)$. We may choose $\alpha\in \R$ so that $\bigabs{\alpha - \tfrac{a}{q}}= \bignorm{\alpha - \tfrac{a}{q}}_\T$. Our task is to bound the Fourier transform of our majorant at $\alpha$.

The first majorant in \eqref{majorants2} has Fourier transform $\hat{\nu}(b_1\alpha)\hat{\nu}(b_2\alpha)$.   We claim that this is bounded in magnitude by 
\[
\ll q^{-1} N^2\max\set{1, \bigabs{\alpha-\tfrac{a}{q}} N}^{-1}+ N^{3/2}W_1^{1/2}Q^2.
\]
As $1 \leq B \ll 1$, the major arc hypothesis for this majorant follows, since it has $L^1$-norm $\gg N^2$ (unless $N \ll W_1^{O(1)}$). 
To establish the claim it suffices to show that for any non-zero integer $b = O(1)$ we have the bound
\begin{equation}\label{claimed major bound}
\abs{\hat{\nu}(b\alpha)} \ll q^{-1/2} N\max\set{1, \bigabs{\alpha-\tfrac{a}{q}} N}^{-1}+ N^{1/2}W_1Q^2.
\end{equation}
This follows from Lemmas \ref{MajorArcAsymptotic}, \ref{sqroot local weyl} and \ref{linear exp}.

Next we turn to the second majorant in \eqref{majorants2}. Employing Lemma \ref{MajorArcAsymptotic} and Lemma \ref{quad maj asymp}, either $N \ll (W_1W_2)^{O(1)}$ or this majorant has Fourier transform bounded in magnitude by 
\begin{multline*}
W_2^{-1/2}\abs{\E_{r_1\in [q]} e_q\brac{b_1a\brac{\trecip{2} W_1r_1^2 + \xi r_1}}}\abs{ \E_{r_2\in[q]} e_q\brac{b_2aW_2r_2^2}}\times \\ \abs{ \int_1^N e\brac{\brac{\alpha - \tfrac{a}{q}}b_1t_1}\intd t_1}\biggabs{\int_{\eta \sqrt{N}}^{\sqrt{N}} e\brac{\brac{\alpha - \tfrac{a}{q}}b_2t_2^2}\intd t_2}+ O\brac{N^{3/2}W_1Q^2}.
\end{multline*}
By Lemma \ref{linear exp}, the first integral is at most $N\max\bigset{1, \bigabs{\alpha - \tfrac{a}{q}}N}^{-1}$. Applying the trivial bound to the second integral, it suffices to prove the bound
\begin{equation}\label{local product}
\abs{\E_{r_1\in [q]} e_q\brac{b_1a\brac{\trecip{2} W_1r_1^2 + \xi r_1}}}\abs{ \E_{r_2\in[q]} e_q\brac{b_2aW_2r_2^2}} \ll q^{-1}.
\end{equation}
By Lemma \ref{sqroot local weyl}, the left-hand side of \eqref{local product} is zero if $\trecip{2}W_1$ and $q/(b_1,q)$ are not coprime. We may therefore assume that they are coprime.
Since $W_2$ is $w$-smooth and $\trecip{2}W_1$ is divisible by the primorial $\prod_{p\leq w} p$, we must have $\hcf(W_2, q/(b_1,q)) =1$, and so $\hcf(b_2W_2, q) \leq b_1b_2$.  Hence Lemmas \ref{local estimate} and \ref{sqroot local weyl} combine to give the bound
\[
\abs{\E_{r_1\in [q]} e_q\brac{b_1a\brac{\trecip{2} W_1r_1^2 + \xi r_1}}}\abs{ \E_{r_2\in[q]} e_q\brac{b_2aW_2r_2}} \ll_{b_1,b_2} q^{-1}.
\]
The major arc bound \eqref{major arc bound} follows with $K = O(1)$.

We simultaneously analyse the third and fourth majorants in \eqref{majorants}.  Under the assumption of our rational approximation to $\alpha$, we have the lower bound
\[
q^{-1}\max\bigset{1, \bignorm{\alpha-\tfrac{a}{q}}_\T N}^{-1} \geq Q^{-2}
\]
Hence using the trivial bound on the quadratic exponential sum, and  Lemma \ref{linear exp} on the linear exponential sum, we obtain the major arc bound \eqref{major arc bound} unless 
\[
\norm{b_2\alpha}_\T \leq Q^2 W_2^{1/2} N^{-1/2}.
\]
In this situation the triangle inequality implies that
\begin{equation}\label{fraction upper bound}
\norm{b_2a/q}_\T \leq Q^2 W_2^{1/2} N^{-1/2} + QN^{-1}.
\end{equation}
Observe that if $N \ll Q^{O(1)}$ then \eqref{major arc bound} follows trivially.  Assuming that this is not the case, and that it is not the case that $N \ll W_2^{O(1)}$, we deduce from \eqref{fraction upper bound} that $\norm{b_2a/q}_\T < 1/q$. The only way this can happen is if $q \mid b_2$.  It therefore suffices to assume that $q\mid b_2$, so that $q = O(1)$.

In the case of the third majorant, Lemma \ref{MajorArcAsymptotic} and the trivial bound for the linear sum together give an upper bound of the form  
\begin{multline*}
(N/W_2)^{1/2}\abs{\int_0^N e\brac{b_1\brac{\alpha -\tfrac{a}{q}}t} \intd t} + O(W_1N Q^2)\\ \ll N^{3/2}W_2^{-1/2} \max\set{1, \abs{b_1\brac{\alpha -\tfrac{a}{q}}} N}^{-1} + W_1N Q^2.
\end{multline*}
This yields the major arc bound \eqref{major arc bound} with $K = O(1)$.

In the case of the fourth majorant, Lemmas \ref{quad maj asymp} and \ref{integral bound} combine (with the trivial bound for the linear sum) to give an upper bound of the form
\begin{equation*}
(N/W_2) \eta^{-1}\max\set{1, \abs{b_1\brac{\alpha -\tfrac{a}{q}}} N}^{-1}+ N^{1/2}Q^2
\end{equation*}
From this we obtain the major arc bound \eqref{major arc bound} with  $K = \eta^{-O(1)}$.
\end{proof}

\begin{lemma}[Minor arc hypotheses]\label{minor hyp veri}
Let $W_1$ be an even positive integer and $\xi \in [W_1]$ with $\hcf(\xi, W_1) =1$. Define $\nu = \nu_{W_1, \xi} : [N] \to [0, \infty)$ as in \eqref{majorant}.
Let $W_2$ be a positive integer, $\eta \in (0, 1/2]$ and define the interval
$$
I := \sqbrac{\eta(N/W_2)^{1/2}, (N/W_2)^{1/2}}.
$$
Fix non-zero integers $b_1, b_2 = O(1)$ and write $B:= |b_1|+|b_2|$. Consider the following four majorants,  mapping each $n \in [-BN,BN ]$ to one of
\begin{equation}\label{majorants}
\sum_{b_1x +b_2y = n} \nu(x)\nu(y), \quad  \sum_{b_1x + b_2W_2y^2=n} \nu(x)1_I(y), \quad  \sum_{b_1x + b_2y=n}  \nu(x)1_I(y), \quad  \sum_{b_1W_2x^2 + b_2y=n} 1_I(x)1_I(y).
\end{equation}
Then either $N \ll (W_1W_2)^{O(1)}$ or all four majorants satisfy the minor arc hypothesis (Definition \ref{minor hyp}).
\end{lemma}

\begin{proof}
By the convolution identity, the Fourier transform of each of the first three majorants is bounded in magnitude by $|\hat{\nu}(b_1\alpha)||I|$. The result then follows for these majorants using Lemma \ref{RefinedMinor} and the fact that $0< |b_1| \ll 1$.

 Letting $\nu_2$ denote the fourth majorant, suppose that $|\hat{\nu}_2(\alpha)| \geq \delta \norm{\nu_2}_1 $. We have $\norm{\nu_2}_1 \gg N/W_2$, unless $N \ll W_2$.  Hence by the convolution identity
\[
\biggabs{ \sum_{x \in I} e(b_1\alpha W_2x^2) \sum_{y \in I}e(b_2\alpha y)} \gg \delta N/W_2.
\]
Thus both of the following estimates hold
\begin{equation}\label{two estimates}
\abs{\sum_{x \in I} e(b_1\alpha W_2x^2)} \gg \delta \sqrt{N/W_2}\qquad\text{and}\qquad\abs{\sum_{y \in I} e(b_2\alpha y)} \gg \delta\sqrt{N/W_2}.
\end{equation}

Applying Weyl's inequality (Lemma \ref{weyl inequality}) to the first sum in \eqref{two estimates}, we deduce the existence of $q_0 \ll \delta^{-O(1)}$ such that $\norm{q_0b_1 W_2\alpha}_\T \ll \delta^{-O(1)} W_2/N$.  Dividing through by $q_0b_1W_2$ and cancelling common factors, it follows that there exist integers $1\leq a \leq q \ll W_2\delta^{-O(1)}$ with $\hcf(a, q) = 1$ and such that $\bignorm{\alpha - \tfrac{a}{q}}_\T \ll \delta^{-O(1)}/N$. We claim that $q \leq |b_2|$, hence completing our proof.

Applying the linear exponential sum estimate (Lemma \ref{linear exp}) to the second sum in \eqref{two estimates}, we deduce that $\norm{b_2\alpha}_\T\ll \delta^{-1}\sqrt{W_2/N}$, hence by the triangle inequality
\[
\bignorm{\tfrac{b_2a}{q}}_\T \ll \bignorm{\alpha - \tfrac{a}{q}}_\T + \norm{b_2\alpha}_\T \ll \frac{\delta^{-1}W_2^{1/2}}{N^{1/2}} + \frac{\delta^{-O(1)}}{N}.
\]
If $q\nmid b_2 $ then  $\norm{b_2a/q}_\T \geq 1/q$ and so either $q \gg \delta^{-1}\sqrt{N/W_2}$ or $q \gg \delta^{O(1)} N$. Each of these conclusions contradict our bound of $q \ll W_2\delta^{-O(1)}$, unless $N \ll (W_2/\delta)^{O(1)}$. The latter implies that $N \ll W_2^{O(1)}$ or $N\ll \delta^{-O(1)}$. If $N\ll \delta^{-O(1)}$ then the conclusion of the minor arc hypothesis is trivial. We may therefore assume that $q \mid b_2$, which certainly implies that $q \leq |b_2|$ (as required).
\end{proof}

\begin{lemma}[Hua-type hypotheses]\label{hua hyp veri}
Let $W_1$ be an even positive integer and $\xi \in [W_1]$. %with $\hcf(\xi, W_1) =1$.
Define $\nu = \nu_{W_1, \xi} : [N] \to [0, \infty)$ as in \eqref{majorant}.
Let $W_2$ be a positive integer, $\eta \in (0, 1/2]$ and define the interval
$$
I := \sqbrac{\eta(N/W_2)^{1/2}, (N/W_2)^{1/2}}.
$$
Fix non-zero integers $b_1, b_2 = O(1)$ and write $B:= |b_1|+|b_2|$. Consider the following four majorants,  mapping each $n \in [-BN,BN ]$ to one of
\begin{equation}\label{majorants3}
\sum_{b_1x +b_2y = n} \nu(x)\nu(y), \quad  \sum_{b_1x + b_2W_2y^2=n} \nu(x)1_I(y), \quad  \sum_{b_1x + b_2y=n}  \nu(x)1_I(y), \quad  \sum_{b_1W_2x^2 + b_2y=n} 1_I(x)1_I(y).
\end{equation}
Then either $N \ll \brac{W_1W_2}^{O(1/\eps)}$ or all four majorants satisfy the Hua-type hypothesis (Definition \ref{hua hyp}) with exponent $\eps$.\end{lemma}

\begin{proof}
We observe that either $N \ll W_1^{O(1)}$ or we have the following estimates
\begin{equation}\label{order of nu}
\norm{\nu}_1 \asymp N \quad\text{and}\quad \norm{\nu}_\infty \ll \sqrt{NW_1}
\end{equation}
We also observe the standard divisor-type estimate: for $n\in \Z\setminus\set{0}$ we have
\begin{equation}\label{divisor}
\sum_{x^2 - y^2 = n} 1 \ll_\eps |n|^\eps.
\end{equation}

We begin with the first majorant in \eqref{majorants3}. In this case, the square of the $L^2$-norm is equal to the count
\[
\sum_{b_1n_1+b_2n_2 = b_1n_3 + b_2n_4} \nu(n_1)\nu(n_2)\nu(n_3)\nu(n_4).
\]
The diagonal contribution to this count, when $n_1 = n_3$, is at most
\[
\brac{\sum_n \nu(n)^2}^2 \ll \norm{\nu}_\infty^2 \norm{\nu}_1^2 \ll W_1N^3 \asymp W_1\norm{\nu}_1^4  N^{-1}.
\]
Using Cauchy--Schwarz and the divisor-bound \eqref{divisor}, the non-diagonal count is given by
\begin{multline*}
\sum_{0< |n| < BN}\brac{\sum_{b_1(n_1-n_3) = n} \nu(n_1)\nu(n_3)}\brac{\sum_{b_2(n_4-n_2) = n} \nu(n_2)\nu(n_4)}\\ \ll (W_1 N)^2 \sum_{0< |n| < BN}\brac{\sum_{(W_1x+\xi)^2-(W_1y+\xi)^2 = 2W_1n} 1}^2\\ \ll_\eps (W_1 N)^2 N (W_1N)^\eps \ll W_1^{O(1)} \norm{\nu}_1^4 N^{\eps -1}.
\end{multline*}
We conclude that for any $\eps >0$ our majorant's second moment has an upper bound  of the form 
$$
O_\eps\brac{W_1^{O(1)} \norm{\nu}_1^4 N^{\eps -1}}.
$$
On assuming that it is not the case that $N \ll_\eps W_1^{O(1/\eps)}$, this can be replaced by an upper bound of the form
$$
\norm{\nu}_1^4 N^{2\eps -1}.
$$
This establishes the lemma for the first majorant.

We turn now to the second majorant in \eqref{majorants3}. In this case the square of the $L^2$-norm is equal to
\[
\sum_{b_1n_1+b_2W_2y_1^2 = b_1n_2 + b_2W_2y_2^2} \nu(n_1)\nu(n_2)1_I(y_1)1_I(y_2).
\]
Provided that $N \geq 2W_2$, the diagonal contribution to this count is at most
\begin{equation}\label{diag for 2}
\sum_n \nu(n)^2\sum_y 1_I(y) \ll \norm{\nu}_\infty \norm{\nu}_1 |I| \ll  W_1 N^{2} W_2^{-1/2}\ll W_1W_2^{1/2}\norm{\nu}_1^2|I|^2  N^{-1}.
\end{equation}
Using the divisor-bound \eqref{divisor}, the non-diagonal count is given by
\begin{multline*}
\sum_{0< |n| < BN}\ \sum_{b_1(n_1-n_2) = n} \nu(n_1)\nu(n_2)\sum_{b_2W_2(y_1^2 - y_2^2) = n} 1_I(y_1) 1_I(y_2) \ll_\eps (W_1W_2)^{O(1)} N^{2+\eps}\\ \ll (W_1W_2)^{O(1)} \norm{\nu}_1^2|I|^2N^{\eps-1}.
\end{multline*}
Using a similar argument to before, this establishes the result for the second majorant.

For the third majorant in \eqref{majorants3}, the diagonal contribution is the same as that in \eqref{diag for 2}. The non-diagonal count is given by
\begin{equation*}
\sum_{0< |n| < BN}\ \sum_{b_1(n_1-n_2) = n} \nu(n_1)\nu(n_2)\sum_{b_2(y_1 - y_2) = n} 1_I(y_1) 1_I(y_2).
\end{equation*}
Notice that if $y_1, y_2 \in I$ then $|y_1 - y_2| < (N/W_2)^{1/2}$. Hence the non-diagonal count is in fact equal to 
\begin{multline*}
\sum_{0< |n| < B(N/W_2)^{1/2}}\ \sum_{b_1(n_1-n_2) = n} \nu(n_1)\nu(n_2)\sum_{b_2(y_1 - y_2 )= n} 1_I(y_1) 1_I(y_2)\\ \ll_\eps (N/W_2)^{1/2} NW_1 N^\eps (N/W_2)^{1/2} \ll (W_1W_2)^{O(1)}  \norm{\nu}_1^2|I|^2N^{\eps-1}.
\end{multline*}
This establishes the result for the third majorant.

For the fourth majorant in \eqref{majorants3}, the diagonal contribution is given by
\begin{equation*}
\brac{\sum_y 1_I(y) }^2 =  |I|^2 = W_2^{O(1)} |I|^4 N^{-1}.
\end{equation*}
 The non-diagonal count is given by
\begin{equation*}
\sum_{0< |n| < B(N/W_2)^{1/2}}\ \sum_{b_1W_2(x_1^2-x_2^2) = n} 1_I(y_1) 1_I(y_2)\sum_{b_2(y_1 - y_2) = n} 1_I(y_1) 1_I(y_2) \ll_\eps (N/W_2)^{1+\eps}\\ \ll W_2^{O(1)} |I|^4 N^{-1} .
\end{equation*}
This establishes the result for the fourth majorant.
\end{proof}

\section{Controlling the counting operator}\label{control section}

 The purpose of this section is to prove an analogue of the Fourier control lemma (Lemma \ref{linear fourier control}) for the counting operator encountered in the quadratic counting  theorem (Theorem \ref{quadratic form satisfies rado}).

Before embarking on this section the reader may wish to recall the definition of $\nu = \nu_{W, \xi}$ (Definition \ref{square majorant}), as well as our notation for the Fourier transform (Definition \ref{fourier transform}) and quadratic Fourier transform (Definition \ref{quadratic fourier transform}).

\begin{lemma}[Mixed restriction estimates]\label{mixed restriction} Let $W_1$ and $W_2$ be $w$-smooth positive integers such that $W_1$ is divisible by $2 \prod_{p \leq w} p$. Given $\xi \in [W_1]$ with $\hcf(\xi, W_1) =1$, define $\nu = \nu_{W_1, \xi} : [N] \to [0, \infty)$ as in \eqref{majorant}.
Given $\eta \in (0, 1/2]$, define the interval
$$
I := \sqbrac{\eta(N/W_2)^{1/2}, (N/W_2)^{1/2}}.
$$
Let $p > 2$ and fix non-zero integers $b_1, b_2 = O(1)$. Then either $N \ll_p (W_1W_2)^{O_p(1)}$ or, for any $f : [N]\to \C$ with $|f| \leq 1_{[N]} + \nu$ and any $B \subset [(N/W_2)^{1/2}]$, we have 
\begin{equation*}
\int_{\T} \bigabs{\hat{f }(b_1\alpha) \hat{1}_B(b_2\alpha)}^p \intd\alpha, \
\int_{\T} \bigabs{\hat{f }(b_1\alpha) \tilde{1}_B(b_2W_2\alpha)}^p \intd\alpha
  \ll_{p} N^{\frac{3p}{2} - 1} W_2^{-\frac{p}{2}},
  \end{equation*}
whilst
\[
 \int_{\T} \bigabs{\tilde{1}_B(b_1W_2\alpha)\hat{1}_B(b_2\alpha)}^p \intd\alpha \ll_{p,\eta} N^{p -1}W_2^{-p}\quad \text{and}\quad \int_{\T} \bigabs{\hat{f }(\alpha)}^{2p}  \intd\alpha \ll_{p}N^{2p-1}.
\]
\end{lemma}

\begin{proof}
Let us first suppose that $|f| \leq \nu$. Then the bounds follow from the abstract restriction estimate (Lemma \ref{abstract restriction}) in conjunction with the verification of the major/minor/Hua-type hypotheses (Lemmas \ref{minor hyp veri}, \ref{major hyp veri}, \ref{hua hyp veri}).

Next let us suppose that $|f| \leq 1_{[N]}$.  We estimate $\bigabs{\hat{1}_B}^p$ and $\bigabs{\tilde{1}_B}^p$ using the trivial bound of $\ll (N/W_2)^{p/2}$. We estimate $\bigabs{\hat{f }}^{p-2}$ using the trivial bound of $ N^{p-2}$, and $\bigabs{\hat{f }}^{2p-2}$ with $N^{2p-2}$. Finally we employ Parseval to give the bound
\[
\int_{\T} \bigabs{\hat{f }(b_1\alpha)}^2 \intd\alpha = \sum_n |f(n)|^2 \leq N.
\]
Combining these inequalities gives the claimed bounds.

Finally, we assume the general bound $|f| \leq 1_{[N]} + \nu$. Write $f =  \theta|f|$, where $|\theta(n)| \leq 1$ for all $n$. Put $f_1 := \theta \min\set{|f|, 1_{[N]}}$ and $f_2 := f - f_1$. Then $f=f_1 + f_2$ with $|f_1| \leq 1_{[N]}$ and $|f_2| \leq \nu$.  Applying the triangle inequality, the estimates now follow from our previous arguments.
\end{proof}

\begin{lemma}[Fourier control]\label{fourier control}
For each $i = 1, 2, 3$, let $L_i$ denote a non-singular linear form in $s_i$ variables with $s_1 \geq 2$, $s_1 + s_2 \geq 3$ and $s_1+s_2+s_3 \geq 5$ (we allow for $s_2 = 0$ or $s_3 = 0$).  Let $W_1$ and $W_2$ be $w$-smooth positive integers such that $W_1$ is divisible by $2 \prod_{p \leq w} p$. Given $\xi \in [W_1]$ with $\hcf(\xi, W_1) =1$, define $\nu = \nu_{W_1, \xi} : [N] \to [0, \infty)$ as in \eqref{majorant}.
Given $\eta \in (0, 1/2]$, define the interval
$$
I := \sqbrac{\eta(N/W_2)^{1/2}, (N/W_2)^{1/2}}.
$$
Suppose that either $W_2=1$ or $s_3 > 0$.  Then either $N \ll (W_1W_2)^{O(1)}$ or for any $f_1, \dots, f_{s_1} : \Z \to \C$,  each satisfying 
$
|f_i| \leq 1_{[N]}+ \nu $, and  any $B \subset I$ we have 
\begin{multline*}
\abs{\sum_{L_1(x) =W_2L_2(y^2) + L_3(z)} f_1(x_1)\dotsm f_{s_1}(x_{s_1})1_{B}(y_1)\dotsm 1_{B}(y_{s_2}) 1_{B}(z_1)\dotsm 1_{B_j}(z_{s_3})}\\ \ll_{\eta} N^{s_1 + \frac12(s_2 + s_3) -1}W_2^{-\frac12(s_2 + s_3)} \min_i\brac{\frac{\bignorm{\hat{f}_i}_\infty}{N}}^{1/10}.
\end{multline*}
\end{lemma}

\begin{proof}
Write 
\begin{equation*}%\label{form notation}
L_i(x) = c_1^{(i)} x_1 + \dots + c_{s_i}^{(i)} x_{s_i}.
\end{equation*}
\\[5pt]
\textbf{\underline{Case 1: $s_3=0$.}}
\\[5pt]
In this case our assumptions imply that $W_2 = 1$. The orthogonality relations then show that our counting operator is equal to
\begin{equation*}
\sum_{L_1(x) =L_2(y^2)} f_1(x_1)\dotsm f_{s_1}(x_{s_1})1_{B}(y_1)\dotsm 1_{B}(y_{s_2}) \\ = \int_\T \prod_i\hat{f}_i\brac{c_i^{(1)}\alpha} \prod_j \tilde1_B\brac{c_j^{(2)}\alpha}\intd\alpha.
\end{equation*}
We apply H\"older's inequality to bound the Fourier integral by
\begin{equation}\label{holder result}
\bignorm{\tilde 1_B}_{s_1+s_2}^{s_2}\bignorm{\hat f_{i}}_\infty^{0.1}\bignorm{\hat f_{i}}_{0.9(s_1+s_2)}^{0.9} \prod_{j \neq i}\bignorm{\hat f_{j}}_{s_1+s_2}.
\end{equation}
Since $s_1 + s_2 \geq 5$, Bourgain's restriction estimate \cite{BourgainLambda} gives that $\bignorm{\tilde 1_B}_{s_1+s_2} \ll N^{\frac{1}{2}-\frac{1}{s_1+s_2}}$. Since $0.9(s_1+s_2) > 4$, Lemma \ref{mixed restriction} gives that
\[
\bignorm{\hat f_{i}}_{0.9(s_1+s_2)}^{0.9} \ll_\eta N^{0.9 - \frac{1}{s_1+s_2}} \quad \text{and}\quad \bignorm{\hat f_{j}}_{s_1+s_2} \ll_\eta N^{1 - \frac{1}{s_1+s_2}}.
\]
The claimed bound follows on incorporating these estimates into \eqref{holder result}.
\\[5pt]
\textbf{\underline{Case 2: $s_3\geq 1$.}}
\\[5pt]
In this case we must assume that $W_2$ is arbitrary.   The orthogonality relations show our counting operator equals 
\begin{multline}\label{fourier integral}
\sum_{L_1(x) =W_2L_2(y^2) + L_3(z)} f_1(x_1)\dotsm f_{s_1}(x_{s_1})1_{B}(y_1)\dotsm 1_{B}(y_{s_2}) 1_B(z_1) \dotsm 1_B(z_{s_3})\\ = \int_\T \prod_i\hat{f}_i\brac{c_i^{(1)}\alpha} \prod_j \tilde1_B\brac{W_2c_j^{(2)}\alpha}\prod_k\hat1_B\brac{c_k^{(3)}\alpha}\intd\alpha.
\end{multline}
We break into further subcases.  Notice that our assumptions that $s_1 \geq 2$, $s_1 +s_2 \geq 3$ and $s_1+s_2 + s_3 \geq 5$ imply that we are in one of the following five situations.
\\[5pt]
\textbf{\underline{Case 2a: $s_1 \geq 4$, $s_3 \geq 1$.}}
\\[5pt]
Fix distinct $i, j \in \set{1,\dots,s_1}$.  Applying the trivial estimate to $\tilde{1}_B$ and all but one copy of $\hat{1}_B$, H\"older's inequality shows that the Fourier integral \eqref{fourier integral} is at most 
\begin{equation*}%\label{holder result}
(N/W_2)^{\frac12\brac{s_2+s_3-1}}\bignorm{\hat f_{i}}_\infty^{0.1}\bignorm{\hat f_{i}}_{0.9(s_1+1)}^{0.9}\bignorm{\hat{f}_{j}\bigbrac{c_{j}^{(1)}\alpha} \hat{1}_B\bigbrac{c_1^{(3)}\alpha}}_{(s_1+1)/2} \prod_{k \notin\set{ i,j}}\bignorm{\hat f_{k}}_{(s_1+1)} .
\end{equation*}
Since $0.9(s_1+1) \geq 4.5$ and $(s_1+1)/2 \geq 2.5$, Lemma \ref{mixed restriction} gives that
\begin{multline*}
\bignorm{\hat f_{i}}_{0.9(s_1+1)}^{0.9} \ll_\eta N^{0.9 - \frac{1}{s_1+1}}, \quad  \bignorm{\hat f_{k}}_{s_1+1} \ll_\eta N^{1 - \frac{1}{s_1+1}},\\ \bignorm{\hat{f}_{j}\bigbrac{c_{j}^{(1)}\alpha} \hat{1}_B\bigbrac{c_1^{(3)}\alpha}}_{(s_1+1)/2} \ll_\eta N^{\frac{3}{2} - \frac{2}{s_1+1}}W_2^{-1/2} .
\end{multline*}
The claimed bound follows.
\\[5pt]
\textbf{\underline{Case 2b: $s_1 = 3$, $s_2 = 0$, $s_3 \geq 2$.}}
\\[5pt]
Let $\set{i, j, k} = \set{1,2,3}$. Applying the trivial estimate to all but two copies of $\hat{1}_B$, H\"older's inequality shows that the Fourier integral \eqref{fourier integral} is at most 
\begin{equation*}%\label{holder result}
(N/W_2)^{\frac12\brac{s_3-2}}\bignorm{\hat f_{i}}_\infty^{0.1}\bignorm{\hat f_{i}}_{4.5}^{0.9}\bignorm{\hat{f}_{j}\bigbrac{c_{j}^{(1)}\alpha} \hat{1}_B\bigbrac{c_1^{(3)}\alpha}}_{2.5} \bignorm{\hat{f}_{k}\bigbrac{c_{k}^{(1)}\alpha} \hat{1}_B\bigbrac{c_2^{(3)}\alpha}}_{2.5} .
\end{equation*}
The claimed bound follows again on employing Lemma \ref{mixed restriction}.
\\[5pt]
\textbf{\underline{Case 2c: $s_1 = 3$, $s_2 \geq 1$, $s_3 \geq 1$.}}
\\[5pt]
Let $\set{i, j, k} = \set{1,2,3}$. Applying the trivial estimate to all but one copy of $\tilde{1}_B$ and all but one copy of $\hat{1}_B$, H\"older's inequality shows that the Fourier integral \eqref{fourier integral} is at most 
\begin{equation*}%\label{holder result}
(N/W_2)^{\frac12\brac{s_2+s_3-2}}\bignorm{\hat f_{i}}_\infty^{0.1}\bignorm{\hat f_{i}}_{4.5}^{0.9}\bignorm{\hat{f}_{j}\bigbrac{c_{j}^{(1)}\alpha} \tilde{1}_B\bigbrac{W_2c_1^{(2)}\alpha}}_{2.5} \bignorm{\hat{f}_{k}\bigbrac{c_{k}^{(1)}\alpha} \hat{1}_B\bigbrac{c_1^{(3)}\alpha}}_{2.5} .
\end{equation*}
The claimed bound follows from Lemma \ref{mixed restriction}.
\\[5pt]
\textbf{\underline{Case 2d: $s_1 = 2$, $s_2 = 1$, $s_3 \geq 2$.}}
\\[5pt]
Let $\set{i, j} = \set{1,2}$. We apply the trivial estimate to all but two copies of  $\hat{1}_B$. H\"older's inequality then shows that the Fourier integral \eqref{fourier integral} is at most 
\begin{equation*}%\label{holder result}
(N/W_2)^{\frac12\brac{s_3-2}}\bignorm{\hat f_{i}}_\infty^{0.1}\bignorm{\hat f_{i}}_{4.5}^{0.9}\bignorm{ \tilde{1}_B\bigbrac{W_2c_1^{(2)}\alpha} \hat{1}_B\bigbrac{c_1^{(3)}\alpha}}_{2.5} \bignorm{\hat{f}_{j}\bigbrac{c_{j}^{(1)}\alpha} \hat{1}_B\bigbrac{c_2^{(3)}\alpha}}_{2.5} .
\end{equation*}
The claimed bound then follows from Lemma \ref{mixed restriction}.
\\[5pt]
\textbf{\underline{Case 2e: $s_1 = 2$, $s_2 \geq 2$, $s_3 \geq 1$.}}
\\[5pt]
Let $\set{i, j} = \set{1,2}$. We apply the trivial estimate to all but two copies of  $\tilde{1}_B$ and all but one copy of $\hat{1}_B$. H\"older's inequality then shows that the Fourier integral \eqref{fourier integral} is at most 
\begin{equation*}%\label{holder result}
(N/W_2)^{\tfrac12\brac{s_3-2}}\bignorm{\hat f_{i}}_\infty^{0.1}\bignorm{\hat f_{i}}_{4.5}^{0.9}\bignorm{ \tilde{1}_B\bigbrac{W_2c_1^{(2)}\alpha} \hat{1}_B\bigbrac{c_1^{(3)}\alpha}}_{2.5} \bignorm{\hat{f}_{j}\bigbrac{c_{j}^{(1)}\alpha} \tilde{1}_B\bigbrac{W_2c_2^{(2)}\alpha}}_{2.5} .
\end{equation*}
The claimed bound  follows from Lemma \ref{mixed restriction}.
\end{proof}

\begin{lemma}\label{mixed mean value}
Let $a_1, a_2, b_1, b_2 \in \Z\setminus \set{0}$ and let $\eta \in (0,1)$.  Given functions $f : (\eta N, N] \to [-1,1]$ and $g : [N] \to [-1, 1]$ we have the bound
$$
\int_\T \abs{\tilde{f}(a_1\alpha) \tilde{f}(a_2\alpha)\hat{g}(b_1\alpha) \hat{g}(b_2\alpha)} \intd\alpha \ll_{b_1, b_2} \eta^{-1} N^2.
$$
\end{lemma}

\begin{proof}
By Cauchy--Schwarz it suffices to bound an integral of the form
$$
\int_\T \bigabs{\tilde{f}(a\alpha) \hat{g}(b\alpha) }^2 \intd\alpha 
$$
for some non-zero integers $a,b$. By orthogonality, this is at most the number of solutions to the equation
\begin{equation}\label{hua type equation}
a(x_1^2- x_2^2) = b(y_1-y_2), \qquad \brac{ x_i \in (\eta N, N],\ y_j \in [N]}.
\end{equation}
The diagonal contribution (when $y_1 = y_2$)  yields at most $N^2$ solutions.  Fix distinct $y_1, y_2 \in [N]$.  Then any solution $(x_1, x_2)$ to \eqref{hua type equation} satisfies
$$
|x_1-x_2| = \frac{|b||y_1-y_2|}{|a|(x_1+x_2)} \leq b \eta^{-1}.
$$
The estimate follows.
\end{proof}

\begin{lemma}[$L^1$ control]\label{l1 control}
Let $a_1, \dots, a_r \in \Z\setminus\set{0}$, $b_1, \dots, b_s \in \Z\setminus\set{0}$ and $c_1, \dots, c_t \in \Z\setminus \set{0}$.  Suppose that  
\[
r \geq 2,\quad r+s \geq 3,\quad s+t \geq 1,\quad  r+s+t \geq 5.
\]  
Then for any $ B \subset [N]$ and $\eta \in (0, 1)$ we have
\begin{equation}\label{2nd estimate}
\sum_{\sum_ia_ix_i^2= \sum_jb_jy_j^2 +\sum_k c_kz_k} \prod_i 1_{(\eta N, N]}(x_i) \prod_j 1_{B}(y_j)\prod_k 1_B(z_k)\\ \ll_{c_i} \eta^{-O(1)} N^{r+s+t -2} \brac{\frac{|B|}{N}}^{1/2}.
\end{equation}
\end{lemma}

\begin{proof}
The left-hand side of \eqref{2nd estimate} can be written as the Fourier integral
\[
\int_\T \prod_i\tilde{1}_{(\eta N, N]}(a_i \alpha) \prod_j \tilde{1}_{B}(b_j\alpha) \prod_k \hat{1}_B(c_k \alpha) \intd\alpha.
\]
If $r+s \geq 5$ then \eqref{2nd estimate} follows from extracting $|B|^{1/2}$ from the Fourier integral, then applying   H\"older's inequality and the estimates
\[
\int_\T \Biggabs{\sum_{x \in (\eta N, N]} e\brac{\alpha x^2}}^{4.5}\intd\alpha,\ \int_\T \Biggabs{\sum_{x \in B} e\brac{\alpha x^2}}^{4.5} \intd\alpha \ll N^{2.5}.
\]
These bounds  are a consequence of \cite{BourgainLambda}.

Let us therefore suppose that $r+s \leq 4$, in which case we must have $t \geq 1$.  We divide into two cases.\\[5pt]
\textbf{\underline{Case 1: $t \geq 2$:}}\\[5pt]
Since $r+s \geq 3$, our Fourier integral contains at least three quadratic exponential sums, at least two of which are equal to $\tilde{1}_{(\eta N, N]}$ (since $r\geq 2$). Employing the bounds $1_{(\eta N, N]} \leq 1_{[N]}$ or  $1_B \leq 1_{[N]}$ on the physical side, we may assume that our third quadratic exponential sum is equal to $\tilde{1}_{[N]}$.  Then using the orthogonality relations and H\"older's inequality, we can  bound the left-hand side of \eqref{2nd estimate} by
\begin{multline*}
N^{r+s+t-5} \bignorm{\hat{1}_B}_\infty^{\frac{1}{2}}
 \brac{\int_\T \bigabs{\tilde{1}_{(\eta N, N]}(a_1\alpha)\hat{1}_{B}(c_1\alpha)}^{2}\intd\alpha}^{1/4}\\
 \brac{\int_\T \bigabs{\tilde{1}_{(\eta N, N]}(a_2\alpha)\hat{1}_{B}(c_2\alpha)}^{2}\intd\alpha}^{1/2}
 \brac{\int_\T \bigabs{\tilde{1}_{(\eta N, N]}(\alpha)}^{6}\intd\alpha}^{1/12}
  \brac{\int_\T \bigabs{\tilde{1}_{[N]}(\alpha)}^{6}\intd\alpha}^{1/6}
  .
\end{multline*}
The estimate now follows from Lemma \ref{mixed mean value} and Bourgain's restriction estimate \cite{BourgainLambda}.
\\[5pt]
\textbf{\underline{Case 2: $t = 1$:}}
\\[5pt]
In this case our Fourier integral contains at least four quadratic exponential sums, at least one of which equals $\tilde{1}_{(\eta N, N]}$. Proceeding as in Case 1, the left-hand side of \eqref{2nd estimate} can be bounded by
\begin{multline*}
N^{r+s+t-5} \bignorm{\hat{1}_B}_\infty^{\frac{1}{2}}
 \brac{\int_\T \bigabs{\tilde{1}_{(\eta N, N]}(a_1\alpha)\hat{1}_{B}(c_1\alpha)}^{2}\intd\alpha}^{1/4}\\
 \brac{\int_\T \bigabs{\tilde{1}_{(\eta N, N]}(\alpha)}^{14/3}\intd\alpha}^{3/28}
  \brac{\int_\T \bigabs{\tilde{1}_{[N]}(\alpha)}^{14/3}\intd\alpha}^{9/14}
  .
\end{multline*}
Again the estimate follows from Lemma \ref{mixed mean value} and Bourgain's restriction estimate \cite{BourgainLambda}.
\end{proof}

\section{A quadratic density result}\label{quadratic density sec}

The purpose of this section is to prove the following.
\begin{theorem}[Density--colouring result]\label{quadratic-linear density}
For each $i = 1, 2, 3$, let $L_i$ denote a non-singular linear form in $s_i$ variables with $s_1 \geq 2$, $s_1 + s_2 \geq 3$ and $s_1+s_2+s_3 \geq 5$ (we allow for $s_2 = 0$ or $s_3 = 0$). Suppose that $L_1(1, \dots, 1) = 0$.
Let $\delta >0$ and let $r$ be a positive integer.  Then either $N \ll_{ \delta, r} 1$ or the following holds. For any sets of integers $A_1, \dots, A_r \subset [N]$  each satisfying $|A_i| \geq \delta N$ and for any $r$-colouring $B_1 \cup \dots \cup B_r = [N]$ there exists $B \in \set{B_1, \dots, B_r}$ such that for all $A\in \set{A_1, \dots, A_r}$ we have
\begin{equation*}
\sum_{L_1(x^2) = L_2(y^2) + L_3(z)} \prod_i1_A(x_i) \prod_j1_{B}(y_j)\prod_k1_{B}(z_k)  \gg_{ \delta, r}  N^{s_1 + s_2 + s_3 -2  }.
\end{equation*}
\end{theorem}

Let $\mathcal{R}_w(N)$ denote the set of $w$-rough numbers in $[N]$, that is those integers all of whose prime divisors exceed $w$. We have the following disjoint partition
\[
[N] = \bigcup_{\zeta \text{ is $w$-smooth}} \zeta \cdot \mathcal{R}_w(N/\zeta).
\]
For each $i$ we would like to find $\zeta_i$ which is not too large and satisfies
\begin{equation}\label{smooth part}
\abs{A_i \cap \brac{\zeta_i \cdot \mathcal{R}_w(N/\zeta_i)}} \geq \tfrac{\delta}{2} \abs{\mathcal{R}_w(N/\zeta_i)}.
\end{equation}
By \cite[Lemma A.3]{CLPRado} there are at most $10^wNM^{-1/2}$ elements of $[N]$ divisible by a $w$-smooth number greater than $M$.   It follows that for each $A_i$ there exists a $w$-smooth number $\zeta_i$ satisfying
$$
\zeta_i \ll \delta^{-O(1)}\exp\brac{O(w)}
$$  
and such that \eqref{smooth part} holds.

Define 
\begin{equation}\label{W defn}
W := 4 \zeta_1^2\dotsm \zeta_r^2\prod_{p \leq w} p \qquad \text{and}\qquad W_i := \frac{W}{2\zeta_i^2}.
\end{equation}
Since $W_i$ is  $w$-smooth and divisible by the primorial $\prod_{p\leq w} p$, we can partition $\mathcal{R}_w(N/\zeta_i)$ into congruence classes
$$
\mathcal{R}_w(N/\zeta_i)\cap\brac{\xi\bmod W_i} = (W_i\cdot\Z + \xi ) \cap [N/\zeta_i], \qquad (\xi \in (\Z/W_i\Z)^\times).
$$
By the pigeon-hole principle, there exists $\xi_i \in (\Z/W_i\Z)^\times$ such that 
$$
\abs{A_i \cap \brac{\zeta_i \cdot \brac{(W_i\cdot\Z + \xi_i ) \cap [N/\zeta_i]}}} \geq \tfrac{\delta}{2} \abs{ (W_i\cdot\Z + \xi_i ) \cap [N/\zeta_i]}.
$$
It follows that there exists a set $A_i'$ of integers such that for every $x \in A_i'$ we have $\zeta_i(W_i x + \xi_i) \in A_i$, and moreover we can ensure that 
\begin{equation}\label{lower bound on A_i'}
A_i' \subset \left(\frac{\delta N}{4\zeta_i W_i},\frac{N-\zeta_i\xi_i}{\zeta_i W_i}\right] \quad \text{and} \quad |A_i'| \geq \frac{\delta N}{4\zeta_iW_i} -O(1).
\end{equation}

We define a colouring of $\sqbrac{\frac{N}{W}}$ by setting
$$
B_j' := \set{x \in \N : Wx \in B_j}.
$$
It follows that
\begin{multline}\label{A to A'}
\sum_{L_1(x^2) = L_2(y^2) + L_3(z)} \prod_l1_{A_i}(x_l)  \prod_m1_{B_j}(y_m)\prod_n1_{B_j}(z_n) \geq \\
\sum_{L_1(\recip{2}W_ix^2+\xi_ix) = WL_2(y^2) + L_3(z)}\prod_l1_{A_i'}(x_l)  \prod_m1_{B_j'}(y_m)\prod_n1_{B_j'}(z_n).
\end{multline}
Set 
$$
X:= \frac{N^2}{W},
$$
and let $\nu_i := \nu_{W_i, \xi_i} : [X] \to [0, \infty)$ be as in \eqref{majorant}. The containment in \eqref{lower bound on A_i'} ensures that for every $x \in A_i'$ we have 
\begin{equation}\label{nu bounds}
 \frac{\delta N}{\zeta_i} \ll \nu_i(\trecip{2}W_i x^2 + \xi_i x) \leq  \frac{N}{ \zeta_i}.
\end{equation}
Define 
$$
f_i(n) := \begin{cases} \nu_i(n) & \text{if } n = \trecip{2}W_i x^2 + \xi_i x \text{ for some } x \in A_i',\\
0 & \text{otherwise.}\end{cases}
$$
Then we have that
\begin{multline}\label{A' to f}
\sum_{L_1(\recip{2}W_ix^2+\xi_ix) = WL_2(y^2) + L_3(z)}\prod_l1_{A_i'}(x_l)  \prod_m1_{B_j'}(y_m)\prod_n1_{B_j'}(z_n) \geq \\
\brac{\frac{\zeta_i}{N}}^{s_1}\sum_{L_1(n) = WL_2(y^2) + L_3(z)}\prod_lf_i(n_l)  \prod_m1_{B_j'}(y_m)\prod_n1_{B_j'}(z_n).
\end{multline}

Notice that  \eqref{lower bound on A_i'} and \eqref{nu bounds} give
$$
\sum_{n \in [X]} f_i(n) \gg \frac{\delta N}{\zeta_i}\brac{\frac{\delta N}{\zeta_iW_i} -O(1)},
$$
so that either $N \ll_{\delta, r, w} 1$ or 
$$
\sum_{n \in [X]} f_i(n) \gg \delta^2 X.
$$
Using Lemma \ref{fourier decay} and the dense model lemma recorded in \cite[Theorem 5.1]{PrendivilleFour}, there exists $0 \leq g_i \leq 1_{[X]}$ satisfying
\begin{equation}\label{dense model approx quad}
\bignorm{\hat{f}_i-\hat{g}_i}_{\infty} \ll  (\log w)^{-3/2}X.
\end{equation}
It follows that either $w \ll_\delta 1$ or, on comparing Fourier coefficients at 0, we deduce that $\sum_{x \in [X]} g_i(x) \gg \delta^2 X$.  Thresholding, define
$$
\tilde{A}_i := \set{x \in [X] : g_i(x) \geq c\delta^2},
$$
with $c$ a small positive absolute constant.  The popularity principle \cite[Ex.1.1.4]{TaoVuAdditive} shows that $|\tilde{A}_i| \gg \delta^2 X$. Hence by Theorem \ref{linear density} there exists $\eta \gg_{\delta, r} 1$ and there exists
\[
\tilde{B}_j := B_j' \cap [\eta N/W, N/W]
\]
such that either $N \ll_{\delta, r, w} 1$ or for each $i=1,\dots, r$ we have
$$
\sum_{L_1(n) = WL_2(y^2) + L_3(z)}\prod_l1_{\tilde{A}_i}(n_l)\prod_m1_{\tilde{B}_j}(y_m)\prod_n1_{\tilde{B}_j}(z_n)  \geq \eta X^{s_1+\frac12(s_2+s_3) -1}W^{-\frac12(s_2+s_3)} 
  .
$$
Using our lower bound for $g_i$ on $\tilde{A}_i$ we deduce that
$$
\sum_{L_1(n) = WL_2(y^2) + L_3(z)}\prod_lg_i(n_l)\prod_m1_{\tilde{B}_j}(y_m)\prod_n1_{\tilde{B}_j}(z_n)  \gg_{\delta, r} X^{s_1+\frac12(s_2+s_3) -1}W^{-\frac12(s_2+s_3)} 
  .
$$
By a telescoping identity there exist functions $h_1$, $\dots$, $h_{s_1}$ $\in \set{f_i, g_i, f_i - g_i}$, at least one of which is equal to $f_i-g_i$,  such that
\begin{multline*}
\abs{\sum_{L_1(n) = WL_2(y^2) + L_3(z)}\brac{\prod_lg_i(n_l)-\prod_lf_i(n_l)}\prod_m1_{\tilde{B}_j}(y_m)\prod_n1_{\tilde{B}_j}(z_n)}\ll \\ \abs{\sum_{L_1(n) = WL_2(y^2) + L_3(z)}\prod_lh_l(n_l)\prod_m1_{\tilde{B}_j}(y_m)\prod_n1_{\tilde{B}_j}(z_n) }.
\end{multline*} 
By Lemma \ref{fourier control} and \eqref{dense model approx quad}, either $N \ll_{\delta, r, w} 1$ or the latter quantity is at most
\begin{equation*}
\ll_{\delta, r} \brac{\frac{\bignorm{\hat{f}_i - \hat{g}_i}_\infty}{X}}^{\frac{1}{10}} X^{s_1+\frac12(s_2+s_3) -1}W^{-\frac12(s_2+s_3)} \\
 \ll X^{s_1+\frac12(s_2+s_3) -1}W^{-\frac12(s_2+s_3)} \log^{-3/20}w .
\end{equation*}

It follows that either $w \ll_{\delta, r} 1$ or that
$$
\sum_{L_1(n) = WL_2(y^2) + L_3(z)}\prod_lf_i(n_l)\prod_m1_{\tilde{B}_j}(y_m)\prod_n1_{\tilde{B}_j}(z_n)   \gg_{\delta, r} 
 X^{s_1+\frac12(s_2+s_3) -1}W^{-\frac12(s_2+s_3)}.
$$
Taking $w$ sufficiently large in terms of $\delta$ and $r$, we deduce that either $N \ll_{\delta, r} 1$ or, on recalling \eqref{A to A'} and \eqref{A' to f}, we have
\begin{multline*}
\sum_{L_1(x^2) = L_2(y^2) + L_3(z)} \prod_l1_{A_i}(x_l)  \prod_m1_{B_j}(y_m)\prod_n1_{B_j}(z_n) \gg_{\delta, r} \\
\brac{\frac{\zeta_i}{N}}^{s_1} X^{s_1+\frac12(s_2+s_3) -1}W^{-\frac12(s_2+s_3)} \gg_{\delta, r} N^{s_1+s_2+s_3-2}.
\end{multline*}
This completes the proof of Theorem \ref{quadratic-linear density}.

\section{Deduction of colouring results from density results}\label{colouring section}

\subsection{When the linear form satisfies Rado's criterion}\label{linear colouring subsection}
The purpose of this section is to prove the following strengthening of Theorem \ref{linear form satisfies rado}. To streamline notation, we suppress the dependence of implicit constants on the coefficients $a_i$ and $b_j$.

\begin{theorem}\label{strong linear form satisfies rado}
Let $a_1,\dots, a_s, b_1,\dots, b_t \in \Z\setminus\set{0}$ with $s,t \geq 1$ and suppose that there exists $S \neq \emptyset$ such that $\sum_{i \in S}a_i = 0$. For any positive integers $r$ and $N$, either $N \ll_{r} 1$ or for any colouring $C_1 \cup \dots \cup C_r = [N]$ there exists $1\leq n \leq r$, a colour class $C_j$ and an interval $I$ of length $N^{1/2^{n-1}}$ such that on setting $M:=  N^{1/2^{n}}$ we have
\begin{equation}\label{strengthened equation}
\sum_{a_1x_1 + \dots + a_s x_s = b_1y_1^2 + \dots + b_ty_t^2}\ \prod_{i \in S}1_{C_j\cap I}(x_i)\prod_{i\notin S} 1_{C_j\cap [M]}(x_i)\prod_{i=1}^t1_{C_j\cap[M]}(y_i)\\ \gg_{ r} M^{|S|+s+t-2}.
\end{equation}
\end{theorem}

The utility of this result over Theorem \ref{linear form satisfies rado} is that it can be used to show that non-trivial monochromatic solutions  exist, given any sensible notion of `trivial'. For if the only monochromatic solutions to our equation %\eqref{linear equals quadratic} 
are trivial, then the left-hand side of \eqref{strengthened equation} should\footnote{For instance, any algebraic notion of `trivial' is likely to deliver a power saving in this estimate.} have order $o(M^{|S|+s+t-2})$, which yields a contradiction if $N$ is sufficiently large in terms of $r$.  
\begin{proof}[Proof of Theorem \ref{strong linear form satisfies rado}]
Re-labelling variables, we can write our equation in the form
$$
L_1(x) = L_2(y^2) + L_3(z),
$$
where the $L_i$ are non-singular linear forms in $s_i$ variables satisfying $s_1 + s_3 = s$, $s_1 = |S|$, $s_2 = t \geq 1$ and $L_1(1, \dots, 1) = 0$. In particular, the latter ensures that $s_1 \geq 2$, and so $s_1+s_2 \geq 3$.  It follows that the conditions of Theorem \ref{linear density} are met with $W= 1$.  Let $\eta(\delta, r)$ denote the parameter appearing in this theorem. A little thought shows that this quantity is increasing with $\delta$ and $1/r$, and redefining if necessary, we may assume that $\eta(\delta, r) \leq \min\set{\delta, r^{-1}}$. Set
\begin{equation}\label{bergelson vector sequence}
\delta_n := \begin{cases} 1/r & \text{when } n = 0;\\
			\trecip{2}\eta(\trecip{2}\delta_{n-1}, r) & \text{otherwise}.\end{cases}
\end{equation}
Let us say that a colour class $C_i$ is \emph{good at scale $n$} if 
\[
\abs{C_i \cap \left(N^{1/2^{n+1}}, N^{1/2^{n}}\right] } \geq \delta_n \abs{\Z \cap \left(N^{1/2^{n+1}}, N^{1/2^{n}}\right]}.
\]
We claim that there exists $1\leq n \leq r$ such that if any $C_i$ is good at scale $n$ then it is also good at scale $m= m(i)$ for some $0 \leq m < n$.

If the claim does not hold, then on defining
$$
S_n : = \set{i \in [r] : C_i \text{ is good at scale $n$}},
$$
we have a chain of strictly increasing subsets
$$
\emptyset \neq S_0 \subsetneq (S_0\cup S_1) \subsetneq \dots \subsetneq (S_0\cup \dots \cup S_{r}),
$$ 
the last of which must have size at least $r+1$.  This contradicts the fact that every element in this chain is a subset of $\set{1,2, \dots , r}$.

Given $n$ satisfying our claim, each colour class $C_i$ satisfies the implication 
\begin{multline}\label{bergelson implication}
\abs{C_i \cap \left(N^{1/2^{n+1}}, N^{1/2^{n}}\right] } \geq \delta_n \abs{\Z \cap \left(N^{1/2^{n+1}}, N^{1/2^{n}}\right]} \quad \implies  \\ \exists m= m(i) < n \text{ with } \abs{C_i \cap \left(N^{1/2^{m+1}}, N^{1/2^{m}}\right] } \geq \delta_m \abs{\Z \cap \left(N^{1/2^{m+1}}, N^{1/2^{m}}\right]}.
\end{multline}
Fixing $i \in S_n$, let $m(i) = m$ be such that $m < n$ and $i \in S_m$.  We can partition $(N^{1/2^{m+1}}, N^{1/2^{m}}]$ into consecutive half-open intervals of integers, all of cardinality at most $N^{1/2^{n-1}}$. In this manner, provided that $N$ is sufficiently large in terms of $r$, the pigeonhole-principle yields an interval of integers $I_i$  satisfying
$$
N^{1/2^{n-1}} \geq |I_i| \geq |I_i \cap C_i| \geq \tfrac12 \delta_{m} N^{1/2^{n-1}}\geq \tfrac12 \delta_{n-1} N^{1/2^{n-1}}.
$$
Letting $t_i+1$ denote the smallest integer in $I_i$, define the set
$$
A_i := \set{x \in[N^{1/2^{n-1}}]: x+t_i \in C_i}.
$$
Then $A_i \subset [N^{1/2^{n-1}}]$ and $|A_i| \geq \trecip{2}\delta_{n-1} N^{1/2^{n-1}}$ for all $i \in S_n$.
%Define a colouring
%$$
%\tilde{C}_j := \set{ x \in \sqbrac{N^{1/2^{n}}} : x \in C_j}.
%$$
%Then $\tilde{C}_1 \cup \dots \cup \tilde{C}_r =   \sqbrac{N^{1/2^{n}}}$.

Notice that Theorem \ref{linear density} remains valid if there are less than $r$ sets $A_i$ of density $\delta$ (simply define new sets $A_i$ to all equal $A_1$).   Applying this result, we deduce that there exists $\tilde{C}_j:= C_j \cap \sqbrac{N^{1/2^{n}}}$ such that for all $A_i$ with $i \in S_n$ we have
\begin{multline}\label{bergelson colouring lower bound}
\sum_{L_1(x) = L_2(y^2) + L_3(z)}1_{A_i}(x_1)\dotsm 1_{A_i}(x_{s_1})1_{\tilde{C}_j}(y_1)\dots 1_{\tilde{C}_j}(y_{s_2})1_{\tilde{C}_j}(z_1)\dots 1_{\tilde{C}_j}(z_{s_3}) \\
\geq \eta(\trecip{2}\delta_{n-1}, r) N^{(2s_1+s_2+s_3-2)/2^{n}}.
\end{multline}
Since $s_1, s_2 \geq 1$ we have the estimate
\begin{equation*}
\sum_{L_1(x) = L_2(y^2) + L_3(z)}1_{A_i}(x_1)\dotsm 1_{A_i}(x_{s_1})1_{\tilde{C}_j}(y_1)\dots 1_{\tilde{C}_j}(y_{s_2})1_{\tilde{C}_j}(z_1)\dots 1_{\tilde{C}_j}(z_{s_3})\\
 \leq |\tilde{C}_j|  N^{(2s_1+s_2+s_3-3)/2^{n}}.
\end{equation*}
Therefore
$$
|C_j \cap (N^{1/2^{n+1}}, N^{1/2^{n}}] | \geq \eta(\trecip{2}\delta_{n-1}, r) N^{1/2^{n}} - N^{1/2^{n+1}}.
$$
Hence, provided that $N$ is sufficiently large in terms of $r$, we have
$$
|C_j \cap (N^{1/2^{n+1}}, N^{1/2^{n}}] | \geq \trecip{2}\eta(\trecip{2}\delta_{n-1}, r) |\Z \cap (N^{1/2^{n+1}}, N^{1/2^{n}}] |  .
$$
As $\delta_n = \trecip{2}\eta(\trecip{2}\delta_{n-1}, r)$, we conclude that $j \in S_n$, so we may take $i := j$ in \eqref{bergelson colouring lower bound}, completing the proof of the theorem.
\end{proof}

\subsection{When the quadratic form satisfies Rado's criterion}

The purpose of this subsection is to prove Theorem \ref{quadratic form satisfies rado}.  Again, we suppress dependence of implicit constants on the coefficients $a_i, b_j$ and the number of variables $s,t$.

\begin{proof}[Proof of Theorem \ref{quadratic form satisfies rado}]
Re-labelling variables, we can write our equation in the form
$$
L_1(x^2) = L_2(y^2) + L_3(z),
$$
where the $L_i$ are non-singular linear forms in $s_i$ variables satisfying $s_1 + s_2 = s \geq 3$, $s_1 = |I|$, $s_3 = t $ and $L_1(1, \dots, 1) = 0$. In particular, the latter ensures that $s_1 \geq 2$. We note that we may assume that $s_2 + s_3 \geq 1$, for otherwise Theorem \ref{quadratic-linear density} implies that for any $A \subset [N]$ with $|A| \geq \delta N$ we have
\[
\sum_{L_1(x^2)= 0} \prod_l 1_A(x_l) \gg_\delta N^{s_1-2}.
\] 
This yields Theorem \ref{quadratic form satisfies rado} since every $r$-colouring has a colour class of density at least $1/r$.

Under the assumption that $s_2 + s_3 \geq 1$,  let $C =O(1) $ denote the implicit constant appearing in Lemma \ref{l1 control}, so that for any $B \subset [N]$ and $\eta \in (0,1)$ we have the bound
\begin{equation}\label{specific l2 control}
\sum_{L_1(x^2) = L_2(y^2) + L_3(z)}\prod_l1_{(\eta N, N]}(x_l)\prod_m1_{B}(y_m)\prod_n 1_{B}(z_n)\\
\leq C\eta^{-C} N^{s_1+s_2+s_3-2} (|B|/N)^{1/2}.
\end{equation}  

Let $c_0(\delta, r)$ denote the implicit  constant occurring in the conclusion of Theorem \ref{quadratic-linear density}. Clearly this quantity is increasing with $\delta$ and $r^{-1}$, and we may assume that $c_0(\delta, r) \leq \min\set{\delta, r^{-1}}$. 

Set
$$
\delta_n := \begin{cases} 1/r & \text{if } n = 1;\\
			\brac{\frac{c_0(\delta_{n-1}/2, r)}{C(\delta_{n-1}/2)^{-C}}}^2 & \text{otherwise}.\end{cases}
$$
Define
$$
\epsilon_n(i) : = \begin{cases} 1 & \text{if } |C_i | \geq \delta_nN;\\
			0 & \text{otherwise}.\end{cases}
$$
Since $\delta_{n+1} \leq \delta_n$, the sequence $\epsilon_n = (\epsilon_n(1), \dots, \epsilon_n(r)) \in \set{0,1}^r\setminus \set{0}$ is monotone increasing in each coordinate as $n$ increases. It follows that this sequence cannot be strictly increasing if it has length at least $r+1$.  Hence there exists $1\leq n \leq r$ for which $\epsilon_n = \epsilon_{n+1}$.  In particular, for any $i$ we have the implication
\begin{equation}\label{sparse dense imp}
|C_i | \geq \delta_{n+1} N \quad \implies \quad |C_i| \geq \delta_n N.
\end{equation}

For each $C_i$ satisfying $|C_i| \geq \delta_n N$ we have
$$
|C_i\cap (\trecip{2}\delta_n N, N]| \geq \trecip{2}\delta_n N.
$$
Notice that Theorem \ref{quadratic-linear density} remains valid if there are less than $r$ sets $A_i$ of density $\delta$ (simply define new sets $A_i$ to all equal $A_1$).   We may therefore apply Theorem \ref{quadratic-linear density}, taking our dense sets to be those $C_i\cap (\trecip{2}\delta_n N, N]$ for which $|C_i| \geq \delta_n N$.  We thereby deduce that there exists $C_j$ such that for all $C_i$ satisfying $|C_i| \geq \delta_n N$ we have
\begin{equation}\label{quadratic-linear lower bound}
\sum_{L_1(x^2) = L_2(y^2) + L_3(z)}\prod_l1_{C_i\cap (\recip{2}\delta_n N, N]}(x_l)\prod_m1_{C_j}(y_m)\prod_n 1_{C_j}(z_n)
 \\
\geq c_0(\trecip{2}\delta_n, r) N^{s_1+s_2+s_3 -2}.
\end{equation}
Applying \eqref{specific l2 control}, we conclude that
$$
 C(\delta_n/2)^{-C}\brac{|C_j|/N}^{1/2} \geq c_0(\delta_n/2, r).
$$
By our construction of the sequence $\delta_n$ it follows that $ |C_j| \geq \delta_{n+1} N$, hence by \eqref{sparse dense imp} we conclude  that $ |C_j| \geq \delta_n N$. We may therefore take $i := j$ in \eqref{quadratic-linear lower bound}, completing the proof of the theorem.
\end{proof}

\section{The Moreira--Lindqvist argument}\label{moreira lindqvist}

In this section we complete our characterisation of when equation \eqref{linear equals quadratic} is partition regular (Theorem \ref{criteria}).
The methods we employ to prove Theorem \ref{quadratic form satisfies rado} do not, at present, succeed for all of the equations covered by Theorem \ref{criteria}.  We begin this section by adapting an argument of Moreira to cover those equations of the form \eqref{linear equals quadratic} for which the quadratic coefficients sum to zero, but for whom the number of variables is not sufficient for us to employ Theorem \ref{quadratic form satisfies rado}. The adaptation of Moreira's argument was explained to the author by Sofia Lindqvist. We begin by using this argument to prove Theorem \ref{bad under hindman}, where the idea  is perhaps more transparent.

\begin{proof}[Proof of Theorem \ref{bad under hindman}]
We first observe that Hindman's conjecture (Conjecture \ref{hindman}) implies the existence of infinitely many monochromatic tuples of the form $(x,y, x+y, xy)$. For given a finite list of such tuples, all monochromatic under the same colour, one can introduce finitely many new colours each attached to the $x$ appearing in a tuple. Re-applying Hindman's conjecture, one obtains a monochromatic configuration under this new colouring, and since the new colour classes introduced are all singletons (and the configuration is not), the configuration is monochromatic under the original colouring (and distinct from each tuple in the list).
 
Given an $r$-colouring $c : \N \to [r]$ define a new colouring $\tilde{c}$ by giving all odd numbers the colour $r+1$ and, if $n$ is even, then it receives the colour $c(n/2)$.  Assuming Conjecture \ref{hindman}, there exist infinitely many $\tilde c$-monochromatic tuples of the form $(x,\ y,\ x+y,\ xy)$. Since all elements of this tuple share the same parity, we deduce that every element is even.  It follows that $$(x/2,\ y/2,\ (x+y)/2,\ xy/2)$$ consists of integers which are monochromatic under $c$.  Finally, we observe that
\[
\brac{\frac{x+y}{2}}^2 - \brac{\frac{x}{2}}^2 = \brac{\frac{y}{2}}^2 + \frac{xy}{2}. \qedhere
\]
\end{proof}

\begin{theorem}\label{sum to zero}
Let $a_1, \dots, a_s, b_1, \dots, b_t \in \Z\setminus\set{0}$ with $s, t \geq 1$ and 
\begin{equation}\label{quad sum to zero}
a_1 + \dots + a_s = 0.
\end{equation}
Then in any finite colouring of $\N$ there are infinitely many tuples $(x_1, \dots, x_s, y_1, \dots, y_t)$ which are monochromatic and which solve the equation \eqref{linear equals quadratic}.
\end{theorem}

\begin{proof}
We closely follow the proof of \cite[Corollary 1.7]{MoreiraMonochromatic}.  As is shown in \cite[\S6]{MoreiraMonochromatic}, there are integers $u_1, \dots, u_s$ not all of which are zero and which satisfy
\begin{equation}\label{u quadric}
a_1u_1^2 + \dots + a_su_s^2 = 0 .
\end{equation}
We claim that we may assume that $a_1u_1 + \dots + a_su_s > 0$. If $a_1u_1 + \dots + a_su_s < 0$ then we reverse the sign of all the $u_i$. If $a_1u_1 + \dots + a_su_s = 0$ then reversing the sign of a single non-zero $u_i$ gives $a_1u_1 + \dots + a_su_s \neq 0$ and we proceed as before.

Let $v_1, \dots, v_t$ denote integers satisfying
\[
b_1 v_1 + \dots + b_tv_t = 0.
\]
For instance, one could take $b_i = 0$ for all $i$, but this is a poor choice if one wishes to generate a monochromatic solution to \eqref{linear equals quadratic} in which all variables are distinct.

Set
\begin{equation}
a := 2(a_1u_1 + \dots + a_su_s) \quad \text{and}\quad b:= b_1 + \dots + b_t.
\end{equation}
Given a colouring $c : \N \to [r]$ define 
\[
\tilde{c}(n) := 
\begin{cases} 
c(bn/a) & \text{ if } a \mid n,\\
r+ (n \bmod a)& \text{ otherwise.}
\end{cases}	
\]
Then $\tilde{c}$ is a finite colouring of $\N$.  Applying \cite[Theorem 1.4]{MoreiraMonochromatic}, there exists infinitely many tuples $(x,y,z)$ giving rise to a $\tilde c$-monochromatic configuration of the form
\begin{equation}\label{moreira config}
x,\quad x+y, \quad x+u_1y, \quad \dots \quad , \quad x+u_sy,\quad xy, \quad 
xy + v_1z, \quad \dots\quad , \quad xy+v_tz.
\end{equation}
Since $x \equiv x+y \pmod a$, we must have that $y \equiv 0 \pmod a$.  Since $x \equiv xy \pmod a$, it follows that all the elements of \eqref{moreira config} are divisible by $a$, and that the configuration
\begin{equation}\label{moreira config 2}
\frac{b(x+u_1y)}{a}, \quad \dots \quad , \quad \frac{b(x+u_sy)}{a},\quad  
\frac{b(xy + v_1z)}{a}, \quad \dots\quad , \quad \frac{b(xy+v_tz)}{a}
\end{equation}
is monochromatic under $c$.

Setting
\[
x_i := \frac{b(x+u_iy)}{a} \quad \text{and} \quad y_j := \frac{b(xy+v_jz)}{a}
\]
we obtain a monochromatic solution to the equation \eqref{linear equals quadratic}.
\end{proof}

With this in hand, we are able to complete our proof of Theorem \ref{criteria}. Since Proposition \ref{necessary} establishes the necessity of Di Nasso and Luperi Baglini's criterion, we need only show that the criterion is sufficient for partition regularity. In other words, we wish to show that if  $a_1,\dots, a_s, b_1,\dots, b_t \in \Z\setminus\set{0}$ with $s,t\geq 1$ and one of the following holds
\begin{enumerate}[(1)]
\item there exists $I \neq \emptyset$ with $\sum_{i \in I } a_i = 0$;\label{quad}
\item there exists $I \neq \emptyset$ with $\sum_{i \in I } b_i = 0$.\label{lin}
\end{enumerate}
then the equation
\begin{equation}\label{linear equals quadratic 2}
a_1x_1^2 + \dots + a_sx_s^2 = b_1 y_1 + \dots + b_t y_t
\end{equation} 
is partition regular. According to our formulation of Theorem \ref{criteria}, we may assume that \eqref{linear equals quadratic 2} does not take the form
\begin{equation}\label{bad case 2}
a(x_1^2-x_2^2) = by^2 + cz
\end{equation}
for some non-zero integers $a,b,c$.

Let us first suppose that we are in situation \eqref{lin}. Applying Theorem \ref{linear form satisfies rado} we obtain infinitely many monochromatic solutions by letting $N \to \infty$.

Next let us suppose that we are in situation \eqref{quad}. If $s \geq 3$ and $s+t \geq 5$ then we may employ Theorem \ref{quadratic form satisfies rado}.  Hence we may assume that either $s < 3$ or $s+t < 5$.
Supposing that $s < 3$, condition \eqref{quad} implies that $2 \geq s \geq |I| \geq 2$, so that $I = \set{1,2} = [s]$.  This situation is covered by Theorem \ref{sum to zero}

Finally let us suppose that $s \geq 3$ and $s+t < 5$.  Since $t \geq 1$, we must have $s = 3$ and $t = 1$. If $I = \set{1,2,3} = [s]$ then we are in the situation covered by Theorem \ref{sum to zero}.  We may therefore assume that $|I| = 2$, $s= 3$ and $t=1$. Hence our equation can be written in the form \eqref{bad case 2}, a case we do not have to deal with.  This completes our proof of Theorem \ref{criteria}.

%%% AUTHOR: optional appendix here
%\appendix %% you may comment this out if no Appendix
%\section*{Appendix}
%\section{Improving the constants}
%Material is placed here as needed.

%%% AUTHOR: optional acknowledgments here
\section*{Acknowledgments} %%  you may comment this out if no Ackno
The author thanks Sofia Lindqvist for the arguments of \S\ref{moreira lindqvist}, and Sam Chow for the idea of using Lemma \ref{mixed mean value}.

%%% AUTHOR:
%%% Bibliography goes here. Note that the arXiv cannot process bibtex
%%% or biber bibliographies.  Example of acceptable bibliograpy format:
\bibliographystyle{amsplain}

%% AUTHOR: You can generate such a bibliography from a .bib file by 
%% running pdflatex/bibtex/pdflatex/pdflatex and then pasting the .bbl file
%% between \begin{thebibliography} and \end{bibliography}

%%% AUTHOR: Include a short description of each author following the
%%% structure below. Use the same short tags used previously.  
%%% Use \imageat{} and \imagedot{} instead of "@" and "." in
%%% email addresses-this replaces the symbols with graphics to avoid 
%%% e-mail address harvesting from the .pdf file
\begin{dajauthors}
\begin{authorinfo}[sean]
  Sean Prendiville\\
  Department of Mathematics and Statistics\\
Lancaster University\\
UK\\
  s\imagedot{}prendiville\imageat{}lancaster\imagedot{}ac\imagedot{}uk \\
  \url{https://sites.google.com/view/seanprendiville/}
\end{authorinfo}
\end{dajauthors}

\end{document}